\documentclass[a4paper,12pt]{article}

\usepackage{mathrsfs}
\usepackage{graphicx}
\usepackage{amsmath}
\usepackage{amsfonts}
\usepackage{amssymb}
\usepackage{amsthm}
\usepackage{textcomp}
\usepackage{color}
\usepackage{geometry}
\usepackage{wasysym}
\usepackage{xspace}

\newtheorem{theorem}{Theorem}[section]
\newtheorem{lemma}{Lemma}[section]
\newtheorem{remark}{Remark}[section]

\newtheorem{example}{Example}[section]
\renewcommand{\theequation}{\arabic{section}.\arabic{equation}}
\renewcommand{\thetheorem}{\arabic{section}.\arabic{theorem}}
\renewcommand{\thelemma}{\arabic{section}.\arabic{lemma}}
\renewcommand{\theproposition}{\arabic{section}.\arabic{proposition}}

\renewcommand{\thealgorithm}{\arabic{section}.\arabic{algorithm}}

\newcommand\dd{~{\rm d}}

\newcommand\Ham{\mathcal{H}}
\newcommand\R{\mathbb{R}}
\newcommand\N{\mathbb{N}}
\newcommand\Z{\mathbb{Z}}

\newcommand\D{\nabla}
\newcommand\del{\delta}

\newcommand{\<}{\langle}
\renewcommand{\>}{\rangle}

\newcommand{\rzz}{\mathscr{R}_z}
\newcommand{\rz}[1]{\mathscr{R}_z({#1})}
\newcommand{\rzo}[2]{\mathscr{R}_z^{#1}({#2})}
\newcommand{\rzoR}{\widetilde{\mathscr{R}}_z^{\Omega_R}}
\newcommand{\rzotilR}{\widetilde{\mathscr{R}}^{\Omega_R}_z(y)}
\newcommand{\resprime}{\mathscr{R}^{\Omega'}_z}
\newcommand{\restilde}{\widetilde{\mathscr{R}}^{\Omega_R}_z}

\newcommand{\dis}[2]{{\rm dist}\left(#1,#2\right)}

\title{QM/MM Methods for Crystalline Defects. \\
  Part 1: Locality of the Tight Binding Model}

\author{Huajie Chen and Christoph Ortner\footnote{{\tt
      huajie.chen@warwick.ac.uk} and {\tt c.ortner@warwick.ac.uk}.  Mathematics
    Institute, University of Warwick, Coventry CV47AL, UK.  This work was
    supported by ERC Starting Grant 335120. CO work was also supported by the
    Leverhulme Trust through a Philip Leverhulme Prize and by EPSRC Standard
    Grant EP/J021377/1.}}

\date{}
\begin{document}
\maketitle

\begin{abstract}
  The tight binding model is a minimal electronic structure model for molecular
  modelling and simulation. We show that the total energy in this model can be
  decomposed into site energies, that is, into contributions from each atomic
  site whose influence on their environment decays exponentially. This result
  lays the foundation for a rigorous analysis of QM/MM coupling schemes.
\end{abstract}

\section{Introduction} \label{sec-introduction}
\setcounter{equation}{0}
QM/MM coupling methods are a class of multi-scale schemes in which a quantum
mechanical (QM) simulation is ``embedded'' in a larger molecular mechanics (MM)
simulation. Due to the high computational cost of QM models, schemes of this
type have become an indispensable tool in many scientific disciplines
\cite{bernstein09, csanyi04, gao02, kermode08, ogata01, warshel76}.  The present
work is the first part in a series establishing the mathematical foundation of
QM/MM schemes in the context of materials modelling.


It is pointed out in \cite{csanyi05} that a minimal requirement for a QM model
to be suitable for QM/MM coupling is the {\em strong locality} of forces,
\begin{equation}
  \label{eq:strong_locality}
  \bigg| \frac{\partial f_n}{\partial y_m} \bigg| \to 0 \quad
  \text{``sufficiently rapidly'' as } r_{nm} \to 0,
\end{equation}
where $r_{nm} = |y_n-y_m|$ and $f_n$ denotes the force acting on an atom at
position $y_n$ within a collection of nuclei at positions
$\{y_\ell\} \subset \R^d$.  The condition \eqref{eq:strong_locality} is called
{\em strong locality} to set it apart from the weaker condition of locality of
the density matrix, which is already well understood (see e.g. \cite{benzi13,
  goedecker99} and \S \ref{sec-further-remark}).

To study \eqref{eq:strong_locality} we take a general tight binding method at
finite electronic temperature as a model problem. We prove an even stronger
condition than \eqref{eq:strong_locality}, {\em strong energy locality}: Given a
finite collection of nuclei $y = \{y_\ell\}$, we decompose the total energy
$E = E(y)$ into
\begin{equation}
  \label{eq:intro:energy_decomposition}
  E(y) = \sum_{\ell} E_\ell(y),
\end{equation}
where the {\em site energies} $E_\ell(y)$ are {\em local} in the sense that
\begin{equation}
  \label{eq:intro:E locality}
  \bigg|\frac{\partial E_\ell(y)}{\partial y_n}\bigg| 
  \lesssim e^{-\gamma r_{\ell n}}, \qquad 
  \bigg|\frac{\partial^2 E_\ell(y)}{\partial y_n \partial y_m} \bigg| \lesssim
  e^{-\gamma (r_{\ell n} + r_{\ell m})},
\end{equation}
for some $\gamma > 0$, and analogous results for higher order derivatives. While
the specific form of the decomposition we employ (cf. \S~\ref{sec-site-energy})
is well-known \cite{ercolessi05,finnis03}, the locality result \eqref{eq:intro:E
  locality} is, to the best of our knowledge, new. We are only aware of one
analogous result, for the Thomas--Fermi--von Weizs\"{a}cker model
\cite{nazar14}.

This locality result has a range of consequences, such as: (1) It provides a
strong theoretical justification for the concept of an interatomic
potential. (2) From a purely analytical point of view, there is little (if any)
distinction between the tight binding model and interatomic potential
models. This means that we can apply many of the analytical tools developed for
interatomic potential models, for example: (3) We can rigorously formulate and
analyze models of defects in infinite crystalline solids \cite{ehrlacher13}.
(4) We can extend the construction and analysis of atomistic/continuum
multi-scale schemes. In particular, the Cauchy--Born continuum limit analysis
\cite{ortner13} can be directly applied without additional work.

(5) Our main motivation, however, is to formulate and analyze new QM/MM coupling
schemes for crystal defects. In this endeavour we build on the successful theory
of atomistic/continuum coupling \cite{luskin13}, employing the tools and
language of numerical analysis. The key idea is that, due to
\eqref{eq:intro:energy_decomposition} and \eqref{eq:intro:E locality}, the total
energy can be approximated by
\begin{displaymath}
  E(y) \approx \sum_{\ell \in {\rm QM}} E_\ell(y) 
  + \sum_{\ell \in {\rm MM}} \tilde{E}_\ell(y),
\end{displaymath}
where $\tilde{E}_\ell$ is {\em not an off-the shelf site potential
  (Lennard-Jones, EAM, \dots)} as in previous works on QM/MM coupling, but
instead is a {\em controlled approximation to $E_\ell$}. We show in the
companion papers \cite{chenpreE, chenpreF} that this approach yields new QM/MM
schemes (both energy-based and force-based) with rigorous rates of convergence
in terms of the QM core region size.

\subsection{Outline}
In Section \ref{sec-tb-finite} we focus on {\em finite systems}. We first
present a thorough discussion of {\em real-space} tight binding models and then
establish the results \eqref{eq:intro:energy_decomposition} and
\eqref{eq:intro:E locality} in this context. In Section
\ref{sec-thermodynamic-limits} we then extend the definition of the site energy
as well as the locality results to infinite systems (with an eye to crystal
lattices) via a limiting procedure. In Section \ref{sec-crystal-defects} we
briefly present two applications of the locality results, both in preparation
for Parts 2 and 3 of this series: We extend the crystal defect model and the
convergence analysis for a truncation scheme from \cite{ehrlacher13} to the tight
binding model. Finally, in Section \ref{sec:appl:numerics}, we present some
preliminary numerical tests illustrating our analytical results.

\subsection{Further Remarks} \label{sec-further-remark}
{\bf Tight binding model. } Tight binding models are minimalistic quantum
mechanics type molecular models, used to investigate and predict properties of
molecules and materials in condensed phases. Both in terms of accuracy and
computational cost they are situated between accurate but computationally
expensive {\it ab initio} methods and fast but limited empirical methods. While
tight binding models are interesting in their own right, they also serve as a
convenient toy model for more accurate electronic structure models such as
(Kohn--Sham) density functional theory.

The study of defects in crystals is a field to which tight binding is well
suited, as it is frequently the case that the deviations from ideal bonding are
large enough that empirical potentials are not sufficiently accurate, but the
system size required to isolate the defect (e.g. for dislocations or cracks)
makes the use of {\it ab initio} calculations challenging. A number of studies
have been carried out in which tight binding is applied to simulations of
crystal defects, see e.g. \cite{kohyama94,lee85,nunes96,wang91}.

{\bf Weak versus strong locality. }  By {\em weak locality} we mean that the
electron density matrix has exponentially fast off-diagonal decay.  In the
context of tight binding, this means
\begin{eqnarray}\label{weak-locality-tb}
  \big[\Gamma(y)\big]_{mn}\lesssim e^{-\gamma r_{mn}}
\end{eqnarray}
(see \eqref{density-matrix} for the definition of the tight binding density
matrix $\Gamma(y)$).
In physics, this is often described by the term ``nearsightedness''
\cite{kohn59,prodan05}, which states that the electron properties of insulators
and metals at finite temperature do not depend on perturbations at distant
regions.  This property has, e.g., been exploited to create linear scaling
electronic structure algorithms \cite{baer97, bowler12, goedecker99, ismail99}.

However, the weak locality is not enough to validate a hybrid QM/MM approach
\cite{csanyi05}.  We need the stronger locality condition
\eqref{eq:strong_locality} to guarantee that the QM region is not affected by
the classical particles, and moreover that the forces in the QM region can be
computed to high accuracy by only considering a small QM neighbourhood.  The
decay rate in equation \eqref{eq:intro:E locality} then gives a guide to how
large the QM region needs to be (see also Theorem \ref{th:appl:min-approx} and
\cite{chenpreE,chenpreF}).
	







{\bf Thermodynamic limit. } Thermodynamic limit problems (infinite body limit),
related to our analysis in Section \ref{sec-thermodynamic-limits}, have been
studied at great length in the analysis literature. The monograph \cite{catto98}
gives an extensive account of the major contributions and also presents the
thermodynamic limit problem for the Thomas--Fermi--von Weizs\"{a}cker (TFW)
model for perfect crystals. The thermodynamic limit of the reduced Hartree--Fock
(rHF) is studied in \cite{catto01} for perfect crystals. This literature also
contains many results on the modelling of local defects in crystals in the
framework of the TFW and rHF models, see e.g. \cite{blanc07, cances13b,
  cances08a, cances08b, cances11, cances13a, lahbabi14, lieb77}.


These discussions are restricted to the case where the nuclei are fixed on a
periodic lattice (or with a given local defect). Leaving the positions of the
nuclei free is also a case of great physical and mathematical
interest. Motivated by \cite{ehrlacher13} we present such a model in
\S~\ref{sec:tb model for point defects}, but postpone a complete analysis to
\cite{chenpre_vardef}.

A related problem is the continuum limit of quantum models. The TFW and rHF
models are studied in \cite{blanc02, cances10} where it is shown that, in the
continuum limit, the difference between the energies of the atomistic and
continuum models obtained using the Cauchy-Born rule tends to zero. The tight
binding and Kohn--Sham models are studied in a series of papers \cite{e07, e10,
  e11, e13}, which establish the extension of the Cauchy--Born rule for smoothly
deformed crystals. Our locality result yields an immediate extension of the
analysis of the Cauchy--Born model in the molecular mechanics case
\cite{ortner13}.

\subsection{Notation}
The symbol $\langle\cdot,\cdot\rangle$ denotes an abstract duality
pairing between a Banach space and its dual.  The symbol $|\cdot|$ normally
denotes the Euclidean or Frobenius norm, while $\|\cdot\|$ denotes an operator
norm. The constant $C$ is a generic positive constant that may change from one
line of an estimate to the next. When estimating rates of decay or convergence,
$C$ will always remain independent of the system size, of lattice position or of
test functions. The dependencies of $C$ will normally be clear from the context
or stated explicitly.

\subsection{List of assumptions}
Our analysis requires a number of assumptions on the tight binding model or the
underlying atomistic geometry. For the reader's convenience we list these with
page references and brief summaries:

\def\asL{{\bf L}\xspace}
\def\asHtb{{\bf H.tb}\xspace}
\def\asHloc{{\bf H.loc}\xspace}
\def\asHsym{{\bf H.sym}\xspace}
\def\asHemb{{\bf H.emb}\xspace}

\begin{center}
  \begin{minipage}{0.9\textwidth}
    \begin{tabular}{rrl}
      \asL & p. \pageref{as:asL} & uniform non-interpenetration \\
      \asHtb & p. \pageref{as:asHtb} & locality of Hamiltonian \\
      \asHloc & p. \pageref{as:asHloc} & locality of Hamiltonian derivatives\\
      \asHsym & p. \pageref{as:asHsym} & symmetries of Hamiltonian \\
      {\bf F} & p. \pageref{as:asF} & configuration independent distribution \\
      {\bf U} & p. \pageref{as:asU} & locality of the repulsive potential\\
      \asHemb & p. \pageref{as:asHemb} & connection between the Hamiltonians of two embedded systems\\
      {\bf D} & p. \pageref{as:asD} & homogeneity of the reference configuration outside a defect core
    \end{tabular}
  \end{minipage}
\end{center}

\section{Tight binding model for finite systems} \label{sec-tb-finite}
\setcounter{equation}{0}
We begin by formulating a general tight binding model for a finite system with
$N$ atoms. Let $\Lambda_N$ be an index set with $\#\Lambda_N = N$. An atomic
configuration is described by a map $y : \Lambda_N\rightarrow\R^d$ with
$d \in \N$ denoting the space dimension. (We admit $d \neq 3$ mostly for the
sake of mathematical generality; e.g., this allows us to formulate simplified
in-plane or anti-plane models.)

We say that the map $y$ is a {\it proper configuration} if the atoms do not
accumulate:
\begin{flushleft} \label{as:asL}
  \asL. \quad $\exists~\mathfrak{m}>0$ such that
  ~~$|y(\ell)-y(k)|\geq\mathfrak{m}~~\forall~\ell,k\in\Lambda_N$.
\end{flushleft}
Let $\mathcal{V}^N_{\mathfrak{m}}\subset\big(\mathbb{R}^d\big)^{\Lambda_N}$
denote the subset of all $y \in (\R^d)^{\Lambda_N}$ satisfying \asL.

\subsection{The Hamiltonian matrix} \label{sec-tb-finite-hamiltonian}
In the tight binding formalism one constructs a Hamiltonian matrix $\mathcal{H}$
in an ``atomic-like basis set''
$\{\phi_{\ell\alpha}({\bf r}-y(\ell))\}_{\ell\in\Lambda_N,\alpha\in\Xi}$,
\begin{eqnarray}\label{tb-H-elements-abstract}
  \Big(\mathcal{H}(y)\Big)_{\ell k}^{\alpha\beta} = 
  \int_{\mathbb{R}^d} \phi_{\ell\alpha}({\bf r}-y(\ell))\widehat{\mathcal{H}}(y)
  \phi_{k\beta}({\bf r}-y(k)) \dd {\bf r},
\end{eqnarray}
where $\Xi$ is a small collection of the atomic orbitals per atom (with maximum
size $n_{\Xi}$), and the exact many-body Hamiltonian operator
$\widehat{\mathcal{H}}$ is usually replaced by a parametrized one.
The entries of the Hamiltonian matrix $\mathcal{H}$ depend on the atomic
species, atomic orbitals and on the configuration of nuclei. In practice, they
are often described by empirical functions ({\it empirical tight binding}) which
have been calibrated using experimental results or results from first principle
calculations.

In either case, we can write the Hamiltonian matrix elements as
\begin{eqnarray}\label{tb-H-elements-pre}
  \Big(\mathcal{H}(y)\Big)_{\ell k}^{\alpha\beta}=h_{\ell k}^{\alpha\beta}(y),
\end{eqnarray}
where
$h_{\ell k}^{\alpha\beta}:\mathcal{V}^N_{\mathfrak{m}}\rightarrow\mathbb{R}$ are
functions depending on $\ell,~k,~\alpha$ and $\beta$.


The orbital indices $\alpha$, $\beta$ do not bring any additional insight into
the problem we are studying, while at the same time complicating the notation.
Therefore, we ignore the indices $\alpha$, $\beta$, which is equivalent to
assuming that there is one atomic orbital for each atomic site ($n_{\Xi}$=1).
The Hamiltonian matrix elements then simply become
\begin{eqnarray}\label{tb-H-elements}
\Big(\mathcal{H}(y)\Big)_{\ell k}=h_{\ell k}(y),
\end{eqnarray}
All our results can be generalized to cases with $n_{\Xi}>1$ without
difficulty. The only required modification is outlined in
\ref{sec-appendix-multiorbital}.

We make the following standing assumptions on the functions $h_{\ell k}(y)$,
which we justify below in Remark \ref{rem:discussion_of_HX} and Examples
\ref{example-2.1}, \ref{example-2.2}. Briefly, these assumptions are
  consistent with most tight binding models, with the only exception that we
  assume that Coulomb interactions are screened.

  \begin{flushleft} \asHtb. \label{as:asHtb}  There exist positive
    constants $\bar{h}_0$ and $\gamma_0$ such that, for any
    $y\in\mathcal{V}^N_{\mathfrak{m}}$,
\begin{eqnarray}\label{h-decay-0}
|h_{\ell k}(y)| \leq\bar{h}_0 e^{-\gamma_0|y(\ell)-y(k)|} \quad\forall~\ell,k\in\Lambda_N.
\end{eqnarray}
\end{flushleft}

\begin{flushleft} \asHloc. \label{as:asHloc}
  There exists $\mathfrak{n}\geq 4$ such that
  $h_{\ell k}\in C^{\mathfrak{n}}(\mathcal{V}^N_{\mathfrak{m}})$.
  %
  %
  Moreover, there exist positive constants $\bar{h}_j$ and $\gamma_j$ for
  $1\leq j\leq\mathfrak{n}$, such that
  \begin{eqnarray}\label{h-decay-j}
    \left|\frac{\partial^j h_{\ell k}(y)}{\partial
    [y(m_1)]_{i_1}\cdots\partial 
    [y(m_j)]_{i_j}}\right| \leq \bar{h}_j 
    e^{-\gamma_j\sum_{l=1}^j\left(|y(\ell)-y(m_l)|+|y(k)-y(m_l)|\right)} 
    \quad\forall~\ell,k\in\Lambda_N
  \end{eqnarray}
  with $m_1,\cdots,m_j\in\Lambda_N$ and $1\leq i_1,\dots,i_j\leq d$.
\end{flushleft}

\begin{flushleft} \asHsym. \label{as:asHsym}
  {\bf (i)} {\it (Isometry invariance)} If $y\in\mathcal{V}_{\mathfrak{m}}^N$
  and $g:\R^d\rightarrow\R^d$ is an isometry, 
  then
	\begin{eqnarray}\label{assumption-same-species-isometry}
		h_{\ell k}(y) = h_{\ell k}(g(y)) \quad\forall~\ell,k\in\Lambda_N.
	\end{eqnarray}
		
	{\bf (ii)} {\it (Permutation invariance)} If
	$y\in\mathcal{V}_{\mathfrak{m}}^N$ and
	$\mathcal{G}:\Lambda_N\rightarrow\Lambda_N$ is a permutation (relabelling)
	of $\Lambda_N$, then
	\begin{eqnarray}\label{assumption-same-species-permutation}
		h_{\ell k}(y) = h_{\mathcal{G}^{-1}(\ell)\mathcal{G}^{-1}(k)}(y\circ\mathcal{G}) 
		\quad\forall~\ell,k\in\Lambda_N.
	\end{eqnarray}
\end{flushleft}

\begin{remark} \label{rem:discussion_of_HX} (i) Condition \eqref{h-decay-0}
  indicates that all the matrix elements are bounded by $\bar{h}_0$, which is
  independent of the system size. This is reasonable under the assumption {\bf
    L} and that the number of atomic orbitals per atom in $\Xi$ remains bounded
  as the number of atoms $N$ increases.
	
  (ii) The condition \asHtb postulates exponential decay of the matrix elements
  with respect to the nuclei distance $|y(\ell)-y(k)|$. This is true in all
  tight binding models; as a matter of fact, most formulations employ a finite
  cut-off (zero matrix elements beyond a finite range of internuclear distance).

  (iii) When $j=1$ in {\bf H.loc}, the condition \eqref{h-decay-j}
  becomes
  \begin{eqnarray}\label{h-decay-1}
    \left|\frac{\partial h_{\ell k}(y)}{\partial [y(m)]_i}\right| 
    \leq \bar{h}_1 e^{-\gamma_1(|y(\ell)-y(m)|+|y(k)-y(m)|)} 
    \quad\forall~\ell,k\in\Lambda_N
  \end{eqnarray}
  with $1\leq i\leq d$.
  This states that there is no long-range interactions in the models, so that
  the dependence of the Hamiltonian matrix elements $h_{\ell k}(y)$ on site $m$
  decays exponentially fast to zero.  This assumption is reasonable if one
  assumes that Coulomb interactions are screened.
  

  (iv) In most tight binding models, the atomic orbitals are not orthogonal,
  which gives rise to an overlap matrix
    \begin{eqnarray}\label{matrix-overlap}
      \Big(\mathcal{M}(y)\Big)_{\ell\alpha k\beta} = 
      \int_{\mathbb{R}^d} \phi_{\ell\alpha}({\bf r}-y(\ell))\phi_{k\beta}({\bf r}-y(k)) d{\bf r}.
    \end{eqnarray}
    (In empirical tight binding models, $\mathcal{M}$ may again be given in
    functional form.)

    On transforming the Hamiltonian matrix from a non-orthogonal to an
    orthogonal basis by taking the transformed Hamiltonian
    \begin{eqnarray*}
      \widetilde{\Ham}=\mathcal{M}^{-1/2}\Ham\mathcal{M}^{-1/2},
    \end{eqnarray*}
    we obtain again the identity as overlap matrix. Moreover, following the
    arguments in \cite{benzi13} it is easy to see that, if $\mathcal{M}$ has an
    exponential decay property analogous to \eqref{h-decay-0}, then so does
    $\mathcal{M}^{-1/2}$.  Thus, we see that the decay properties in \asHtb
    and {\bf H.loc} are not lost by this transformation and we can, without loss of
    generality, ignore the overlap matrix.  
    
    (v) We have opted to work with an isolated system, however, it would be
    equally possible to employ periodic boundary conditions. In this case, tight
    binding models employ Bloch sums to take into account the periodic images;
    see, e.g., the Slater--Koster formalism \cite{slater54}. Our entire analysis
    can be easily adapted to this case as well, and the resulting thermodynamic
    limit model would be identical to the one we obtain.
    
    

    

    (vi) {\bf H.sym (i)}, invariance of the Hamiltonian under isometries of
    deformed space, is true in the absence of an external (electric or magnetic)
    field.  E.g., with $g(x)=x+c$, with some $c\in\R^d$, {\bf (i)} implies
    translation invariance of the Hamiltonian.
    {\bf H.sym (ii)} indicates that all atoms of the system belong to the same
    species so that the relabelling of the indices only gives rise to
    permutations of the rows and columns of the Hamiltonian.
    
    The symmetry assumptions {\bf H.sym} are natural and represent no restriction
    of generality. We require them to establish analogous symmetries in the site
    energies that we define in \S~\ref{sec-site-energy}. We remark, however,
    that {\bf H.sym} must be modified for multiple atomic orbitals per site; see
    \ref{sec-appendix-multiorbital}.
\end{remark}

\begin{example}\label{example-2.1} 
  Many tight binding models use the `two-centre approximation'
  \cite{goringe97}, assuming that $h_{\ell k}(y)$ depends only on the
  vector between two atoms $y(\ell)$ and $y(k)$. If we only take into
  accounts the nearest neighbour interactions, then the Hamiltonian
  matrix elements of such models are given by:
  \begin{eqnarray}\label{tb-H-elements-ex2.1}
    \Big(\mathcal{H}(y)\Big)_{\ell k}=\left\{ \begin{array}{ll}
        a_{\ell} & {\rm if}~\ell=k; \\[1ex]
        b_{\ell k}(y(\ell)-y(k)) & {\rm if}~y(\ell)~{\rm
          is~the~nearest~neighbor~of}~y(k); \\[1ex ]
        0 & {\rm otherwise},
      \end{array} \right.
  \end{eqnarray}
  where $a_{\ell}$ are constants and $b_{\ell k}$ are smooth
  functions. We observe that all our assumptions in \asHtb and {\bf
    H.loc} are trivially satisfied for this simple but common model.
\end{example}

\begin{example}\label{example-2.2}
  The Hamiltonian of a reduced Hartree--Fock model with the Yukawa
  potential~\cite{yukawa35} is
\begin{eqnarray}
\widehat{\mathcal{H}}(y)=-\frac{1}{2}\Delta-\sum_{\ell\in\Lambda_N}
Y_m(\cdot-y(\ell))+\int_{\mathbb{R}^d}\rho(x)Y_m(\cdot-x)dx,
\end{eqnarray}
where $\rho$ is assumed to be a fixed electron density, and $Y_m$ is the Yukawa
kernel with parameter $m>0$:
\begin{eqnarray}\label{yukawa-kernel}
Y_m(x)=\left\{ \begin{array}{ll}
m^{-1}e^{-m|x|} & {\rm if}~d=1; \\[1ex]
\int_0^{\infty} e^{-m|x|\cosh t}dt & {\rm if}~d=2; \\[1ex]
|x|^{-1}e^{-m|x|} & {\rm if}~d=3.
\end{array}\right.
\end{eqnarray}
Note that both $Y_m$ and its derivatives decay to 0 exponentially fast.
If the basis functions $\{\phi_{\ell\alpha}\}_{\ell\in\Lambda_N,\alpha\in\Xi}$
are localized, i.e. the atomic orbitals for the $\ell$th atom have compact
support around $y(\ell)$, or decay exponentially, then we have that the matrix
elements generated by \eqref{tb-H-elements-abstract} satisfy the assumptions in \asHtb and {\bf H.loc}.
\end{example}

As a consequence of our assumptions, the following lemma states that the
sepctrum of the Hamiltonian is uniformly bounded with respect to the system size
$N$.

\begin{lemma}\label{lemma-H-bound}
  For any $\Lambda_N$ satisfying \asL and \asHtb, there exist constants
  $\underline{\sigma}$ and $\overline{\sigma}$ depending only on
  $\mathfrak{m},\bar{h}_0, \gamma_0$ and $d$, such that, for all
  $y \in \mathcal{V}_{\frak{m}}^N$,
  \begin{eqnarray}\label{assump-spectrum}
    \sigma(\mathcal{H}(y))\subset [\underline{\sigma},\overline{\sigma}].
  \end{eqnarray}
\end{lemma}

\begin{proof}
	Using \eqref{h-decay-0} and the Ger\v{s}gorin's theorem \cite{horn91}, we have
	\begin{eqnarray}
	|\lambda_i| \leq \max_{\ell\in\Lambda_N}\left( |h_{\ell\ell}(y)| 
	+ \sum_{k\in\Lambda_N,~k\neq\ell}|h_{\ell k}(y)| \right) \leq 
	\bar{h}_0\max_{\ell\in\Lambda_N}\left(\sum_{k\in\Lambda_N}e^{-\gamma_0|y(\ell)-y(k)|}\right)
	\end{eqnarray}
	for any $i$.
	This together with the assumption \asL implies that
	\begin{eqnarray}
	|\lambda_i|\leq \bar{h}_0\max_{\ell\in\Lambda_N}\left(\sum_{k\in\Lambda_N} 
	e^{-\gamma_0\mathfrak{m}|y(\ell)-y(k)|}\right) \leq  \frac{C_d\bar{h}_0}{(\mathfrak{m}\gamma_0)^d}
	\end{eqnarray}
	for any $i$, where $C_d$ is a constant depending only on the dimension $d$.
\end{proof}

\subsection{Band energy} \label{sec-e-f-H}
The total energy of of a configuration $y \in \mathcal{Y}_\mathfrak{m}^N$ is written
as the sum of band energy and repulsive energy,
\begin{eqnarray} \label{e-tot}
  E^{\rm tot}(y)=E^{\rm band}(y)+E^{\rm rep}(y),
\end{eqnarray}
which we define as follows. For simplicity of notation, we will write
$E=E^{\rm band}$ throughout this paper.

Given a deformation $y \in \mathcal{V}_{\frak{m}}$, the associated Hamiltonian
matrix $\mathcal{H}(y)$, and its eigenvalues $\epsilon_s$ and eigenvectors
$\psi_s$, $s = 1, \dots, N$ (allowing for multiplicity),
\begin{eqnarray}\label{eigen-H}
\mathcal{H}(y)\psi_s = \varepsilon_s\psi_s\quad s=1,2,\cdots,N,
\end{eqnarray}
the band energy of the system is defined by
\begin{equation}\label{e-band}
  E(y) = \sum_{s = 1}^N f(\varepsilon_s)\varepsilon_s
  = \sum_{s = 1}^N \mathfrak{f}(\varepsilon_s),
\end{equation}
where $f$ depends on the physical context. For example, at finite electronic
temperature, $f$ is the is Fermi-Dirac function
\begin{eqnarray}\label{fermi-dirac}
  f(\varepsilon) = \bigg( 1+e^{(\varepsilon-\mu)/(k_{\rm B}T)} \bigg)^{-1},
\end{eqnarray}
and $\mu$ is a fixed chemical potential (more on that below). In the
zero-temperature limit, $f$ becomes a step function.  In practical simulation of
conductors, $f$ is often a smearing function (i.e., a numerical parameter) to
ensure numerical stability (see, e.g. \cite{fu83,kresse96,motamarri13}.

In the present work, we shall not be too concerned about the origin of
$\mathfrak{f}$, but simply accept it as a model parameter. Our analysis can be
carried out whenever $f$ is analytic (e.g., the Fermi-Dirac distribution) or, in
insulators (systems with band gap at $\mu$) also when $f$ is a step
function. For the sake of a unified presentation we shall only present the first
case, but it will be immediately apparent how to treat insulators as well. Thus
we shall assume for the remainder of the paper that

\begin{flushleft} {\bf F}. \label{as:asF}
   \quad $f$ is a configuration independent analytic function in an
  open neighbourhood $D_f \subset \mathbb{C}$ of
  $[\underline{\sigma},\overline{\sigma}]$; cf. Lemma \ref{lemma-H-bound}.
\end{flushleft}

\begin{remark}
  (i) The qualifier ``configuration independent'' in {\bf F} essentially
  rephrases the assumption that the chemical potential $\mu$ is independent of
  the configuration $y$. This is false in general, but a reasonable assumption
  in our context since, in the next section, we shall consider limits of finite
  bodies in the form of lattices that are only locally distorted by defects. It
  is well-known (though we are unaware of a rigorous proof) that the limiting
  potential $\mu$ is indeed configuration independent, but is only a function of
  the far-field homogeneous lattice state.
  
  (ii) We note, though, that there is a simple model in which the Fermi-level is
  indeed independent of the configuration. Consider a single-species two-centre
  approximation where $h_{\ell k}(y) = h(|y(\ell) - y(k)|)$ and
  $h_{\ell\ell}(y) \equiv a_0$. Then it is easy to see that the spectrum is
  symmetric about $a_0$ and hence the Fermi level is always $\mu = a_0$.
\end{remark}

The repulsive component of the energy is empirical and in most of the cases is simply
described by a pair potential interaction
\begin{eqnarray}\label{e-rep}
  E^{\rm rep}(y) = \frac{1}{2}\sum_{\ell,k\in\Lambda_N,~\ell\neq k}
  U_{\ell k}\big(y(\ell)-y(k)\big),
\end{eqnarray}
where $U_{\ell k}$ is an empirical repulsive energy acting between atoms on
$y(\ell)$ and $y(k)$. For future reference, we rewrite this in site-energy form,
\begin{eqnarray}\label{e_rep_l}
  E^{\rm rep}(y) = \sum_{\ell \in \Lambda_N} E_\ell^{\rm rep}(y), \qquad
  E_\ell^{\rm rep}(y) = \frac12 \sum_{k\in \Lambda_N,~k \neq \ell} 
  U_{\ell k}\big(y(\ell)-y(k)\big),
\end{eqnarray}
and we shall assume throughout that
\begin{flushleft}
  {\bf U}. \label{as:asU}
  \quad $U_{\ell k} \in C^\frak{n}(\R^d \setminus B_{\frak{m}})$ and there exist
  $c_{U}, \gamma_U > 0$ such that 
  \begin{eqnarray}
  |\D^j U_{\ell k}({\bf r})| \leq c_U \exp(-\gamma_U |{\bf r}|) \qquad\forall~\ell,k\in\Lambda_N
  \end{eqnarray}
  for $0 \leq j \leq \frak{n}$.
\end{flushleft}

In most of our analysis we shall only be concerned with the band energy
$E$, and have added $E^{\rm rep}$ mostly for the sake of
completeness. The pair interaction in $E^{\rm rep}$ may be replaced with an
arbitrary short-ranged interatomic potential.

\subsubsection{Representation via contour integrals}
\label{sec:finite-contour}
Our analysis of the locality of interaction generated by the tight binding model
builds on a representation of $E$ in terms of contour integrals.  The main issue
is to represent the electronic density matrix as an operator-valued function of
the Hamiltonian.  This technique has been used in quantum chemistry, for example
\cite{e10, goedecker95} for tight binding and
\cite{cances14,e11,goedecker93,lin09,yang95} for density functional theory.

We begin by defining, for any proper configuration
$y\in\mathcal{V}^N_{\mathfrak{m}}$, the electronic density matrix (or simply,
{\em density matrix}) of the system,
\begin{eqnarray}\label{density-matrix}
\Gamma(y)=\sum_s f(\varepsilon_s)|\psi_s\rangle\langle\psi_s|
=f(\mathcal{H}(y)).
\end{eqnarray}
The band energy can then equivalently be written as
\begin{eqnarray}
  E(y)={\rm Tr}\big(\mathcal{H}(y)\Gamma(y)\big)
  ={\rm Tr}\big(\mathcal{H}(y)f(\mathcal{H}(y))\big)
  ={\rm Tr}\big(\mathfrak{f}(\mathcal{H}(y))\big).
\end{eqnarray}
Lemma \ref{lemma-H-bound} and {\bf F} imply that we can find a bounded
contour $\mathscr{C} \subset D_f$, circling all the eigenvalues
$\varepsilon_s$ on the real axis
(see Figure \ref{fig-contour}), and satisfies
\begin{eqnarray}\label{dist-c}
  \min\left\{ {\rm dist}\big(\mathscr{C},\sigma(\mathcal{H}(y))\big),
    {\rm dist}\big(\mathscr{C},\mathfrak{s}(\mathfrak{f})\big) \right\} \geq \mathfrak{d},
\end{eqnarray}
with a constant $\mathfrak{d}>0$ that is independent of $y$ or of $N$.
Let
\begin{eqnarray*}
\rz{y}:=\big(\mathcal{H}(y)-zI\big)^{-1}
\end{eqnarray*}
denote the resolvent of $\Ham(y)$, then
\begin{eqnarray}\label{fH-contour}
\mathfrak{f}(\mathcal{H}(y))=-\frac{1}{2\pi i}\oint_{\mathscr{C}}\mathfrak{f}(z)\rz{y}\dd z,
\end{eqnarray}
which implies that
\begin{eqnarray}\label{e-contour}
E(y)=-\frac{1}{2\pi i}\oint_{\mathscr{C}}\mathfrak{f}(z){\rm Tr}\Big(\rz{y}\Big)\dd z.
\end{eqnarray}

\begin{figure}[ht]
	\centering
	\includegraphics[width=10.0cm]{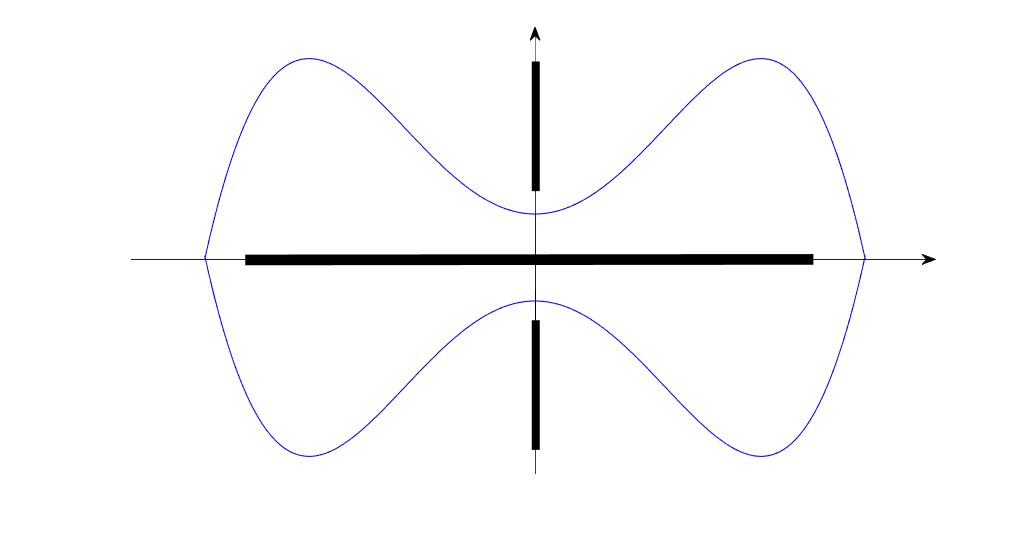}
	\put(-230,110){$\mathscr{C}$}
	\put(-132,110){non-analytic}
	\put(-120,62){spectrum}
	\caption{A schematic plot of the dumbbell-shaped Cauchy contour $\mathscr{C}$.}
	\label{fig-contour}
\end{figure}

It is already clear from \eqref{e-contour} that the locality of the resolvents
will play an important role in our analysis. Hence, we prove a decay estimate in
the next lemma.

\begin{lemma}\label{lemma-resolvant-decay}
  Let $\mathcal{H}(y), y \in \mathcal{V}_{\frak{m}}^N$ be a tight binding
  Hamiltonian of the form \eqref{tb-H-elements} and $\mathscr{C}$ a contour
  satisfying \eqref{dist-c}. If \asL and \asHtb are satisfied, then there exist
  constants $\gamma_{\rm r}>0$ and $c_{\rm r}$, independent of $y$ or $N$, such
  that
  \begin{eqnarray}\label{resolvant-decay}
    \Big(\rz{y}\Big)_{\ell k}\leq c_{\rm r}e^{-\gamma_{\rm r}|y(\ell)-y(k)|} 
    \quad\forall~z\in\mathscr{C}.
  \end{eqnarray}
\end{lemma}

\begin{proof}
  This proof relies on the arguments provided by \cite{e10} and a Combes--Thomas
  type estimate~\cite{combes73}.

  For $k_0$ and $\gamma_{\rm r}>0$, let $B \in \R^{\Lambda_N \times \Lambda_N}$,
  \begin{eqnarray}
    B_{\ell k}
    =\left\{ 
    \begin{array}{ll}
      e^{\gamma_{\rm r}|y(\ell)-y(k_0)|}, & {\rm  if}~\ell=k; \\[1ex]
      0, &  {\rm otherwise}.
    \end{array}
           \right.
  \end{eqnarray}
  From this definition,  we have
  \begin{align*}
    \big[B\mathcal{H}B^{-1} - \mathcal{H} \big]_{\ell k}
    &=
      e^{\gamma_{\rm r}|y(\ell)-y(k_0)|} \mathcal{H}_{\ell k} 
      e^{-\gamma_{\rm r}|y(k)-y(k_0)|} - \mathcal{H}_{\ell k} \\
    &= 
      \mathcal{H}_{\ell k} 
      \big(e^{\gamma_{\rm r}(|y(\ell)-y(k_0)|-|y(k)-y(k_0)|)}-1\big).
  \end{align*}
  Assumptions \asL and \asHtb yield
  \begin{eqnarray} \nonumber
    &&   \hspace{-1.5em}
     \big\|B\mathcal{H}B^{-1}-\mathcal{H}\big\|_{\infty} 
       ~\leq~ 
       \sup_{\ell\in \Lambda_N} \sum_{k\in\Lambda_N}|\mathcal{H}_{\ell k}| 
       \Big(e^{\gamma_{\rm r}|y(\ell)-y(k)|}-1\Big) \\[1ex] \nonumber
    &\leq&
           \bar{h}_0 \sup_{\ell\in \Lambda_N} \left( \bigg(
           \sum_{\substack{k\in\Lambda_N, \\ |y(\ell)-y(k)|> R}}
           e^{-(\gamma_0-\gamma_{\rm r}) |y(\ell)-y(k)|} \bigg) 
           + 
           \bigg(
           \sum_{\substack{k\in\Lambda_N, \\ |y(\ell)-y(k)|\leq R}}
           e^{-\gamma_0 |y(\ell)-y(k)|} \bigg) \big(e^{\gamma_{\rm r} R}-1\big) \right)
    \\[1ex]
    &\leq& 
           C \Big(e^{-\frac{1}{2}(\gamma_0-\gamma_{\rm r})R} + e^{\gamma_{\rm r}R} -1 \Big)
  \end{eqnarray}
  for any $\gamma_{\rm r}<\gamma_0/2$ and $R>0$,
  where $C$ is a constant
  depending only on $\bar{h}_0$, $d$, $\gamma_0$, $\gamma_{\rm r}$, 
  and $\mathfrak{m}$.
  For any $\varepsilon>0$, we can choose $R$ sufficiently large and
  then 
  $\gamma_{\rm r}$ sufficiently small (depending on $R$) such that
  $\|B\mathcal{H}B^{-1}-\mathcal{H}\|_{\infty}<\varepsilon$. We note that 
  the choice of $R$ and $\gamma_{\rm r}$ does not depend on the system size 
  $N$ but only on $\varepsilon$ and the constants
  $\bar{h}_0,~\gamma_0,~\mathfrak{m}$. Similarly, we have the same bound for
  $\|B\mathcal{H}B^{-1}-\mathcal{H}\|_{1}$.  Using interpolation, we get the
  same bound for $\|B\mathcal{H}B^{-1}-\mathcal{H}\|_{2}$.
	
  Note that
  \begin{eqnarray}
    B(z-\mathcal{H})^{-1}B^{-1} &=& (z-B\mathcal{H}B^{-1})^{-1}  \\
    &=& \notag
        (z-H)^{-1}(I-(B\mathcal{H}B^{-1}-\mathcal{H})(z-\mathcal{H})^{-1})^{-1}.
  \end{eqnarray}
  Since \eqref{dist-c} implies
  $\|(z-\mathcal{H})^{-1}\|_{\mathscr{L}(l^2)} ]\leq 1/\mathfrak{d}$, we can
  choose $R$ and $\gamma_{\rm r}$ such that $z-B\mathcal{H}B^{-1}$ is invertible
  and
  \begin{displaymath}
    \|B(z-\mathcal{H})^{-1}B^{-1}\|_{\mathscr{L}(l^2)} \leq \frac{2}{\mathfrak{d}}.
  \end{displaymath}
  Using
  $\big|[B(z-\mathcal{H})^{-1}B^{-1}]_{\ell k}\big| \leq
  \|B(z-\mathcal{H})^{-1}B^{-1}\|_{\mathscr{L}(l^2)} \leq 2 / \frak{d}$ and
  \begin{align*}
    \Big| \big[(z-\mathcal{H})^{-1}\big]_{\ell k} 
    e^{\gamma_{\rm r}(|y(\ell)-y(k_0)|-|y(k)-y(k_0)|)} \Big|
    = 
      \Big| \big[B(z-\mathcal{H})^{-1}B^{-1}\big]_{\ell k} \Big| 
    \leq
      \frac{2}{\frak{d}},
  \end{align*}
  and consequently,
  \begin{eqnarray}
    \left|\big[(z-\mathcal{H})^{-1}\big]_{\ell k}\right|
    \leq 
    \frac{2}{\frak{d}} e^{-\gamma_{\rm r}(|y(\ell)-y(k_0)|-|y(k)-y(k_0)|)}.
  \end{eqnarray}
  Taking $k_0=k$, we obtain the stated exponential decay estimate.
\end{proof}

\subsection{Site energy} \label{sec-site-energy}

Since the tight binding Hamiltonian \eqref{tb-H-elements-abstract} is given in
terms of an atomic-like basis set, we can distribute the energy among atomic
sites. This is a well-known idea, which has been used for constructing
interatomic potentials based on the bond-order concept (see
e.g. \cite{ercolessi05,finnis03,tersoff88}).

Noting that $\|\psi_s \|_{\ell^2} = 1$ for all $s$, we have
\begin{align*}
  E(y) 
  &= \sum_{s = 1}^N f(\varepsilon_s)\varepsilon_s
    =
    \sum_{s = 1}^N f(\varepsilon_s)\varepsilon_s \sum_{\ell \in \Lambda_N} [\psi_s]_\ell^2
  = \sum_{\ell \in \Lambda_N} \sum_{s = 1}^N f(\varepsilon_s)\varepsilon_s [\psi_s]_\ell^2.
\end{align*}
That is, we have obtained the decomposition of the band energy
\begin{align}
  \label{E-El}
  E(y) &= \sum_{\ell\in\Lambda_N}E_{\ell}(y),  \qquad \text{where} \\[0.2em]
  \label{site-energy}
  E_\ell(y) &= \sum_{s} f(\varepsilon_s) \varepsilon_s [\psi_s]_{\ell}^2 
              = \sum_{s}\mathfrak{f}(\varepsilon_s) [\psi_s]_{\ell}^2,
\end{align}
and we call $E_\ell(y)$ the site energy. 

When the atomic orbitals are not orthogonal we slightly modify the definition of
site energy slightly for computational efficiency; see
\ref{sec-appendix-siteoverlap} for detailed discussions.

Such a decomposition is useful, for example, since classical interatomic
potentials almost always decompose the total energy in such a way, hence the
relation \eqref{E-El} can be used to establish a bridge between {\it ab initio}
models and empirical interaction laws \cite{ercolessi05}. For our own purpose,
the decomposition will (1) yield a relatively straightforward thermodynamic
limit argument to define and analyze variational problems on the infinite
lattice along the lines of \cite{ehrlacher13}, and (2) provide a starting point
for the construction and analysis of concurrent multi-scale schemes hybrid
models, which we will pursue in two companion papers \cite{chenpreE, chenpreF}.

Our aim, next, is to establish locality of $E_\ell$.  From now on, for the sake
of readability, we drop the argument $(y)$ in $\rz{y}$, $\mathcal{H}(y)$,
$\mathcal{H}_{,m}(y)$ and $\mathcal{H}_{,mn}(y)$ whenever convenient and
possible without confusion and in addition write $r_{mn} = |y(m) - y(n)|$.
%
Let $e_{\ell}$ be the $N$
dimensional canonical basis vector, then we obtain from \eqref{eigen-H} that
\begin{align*}
  E_{\ell}(y) 
  &= \sum_{s}\mathfrak{f}(\varepsilon_s)\left|[\psi_s]_{\ell}\right|^2 
    =\sum_{s}\mathfrak{f}(\varepsilon_s)(\psi_s,e_{\ell})(\psi_s,e_{\ell}) 
    =\sum_{s}\Big(\mathfrak{f}\big(\mathcal{H}\big)\psi_s,e_{\ell}\Big)
    (\psi_s,e_{\ell}) \\[1ex]
  &= \sum_{s}(\psi_s,e_{\ell})\Big(\psi_s,\mathfrak{f}\big(\mathcal{H}\big)e_{\ell}\Big) 
  = \Big(e_{\ell},\mathfrak{f}\big(\mathcal{H}\big)e_{\ell}\Big),
\end{align*}
and employing \eqref{fH-contour} we arrive at
\begin{eqnarray}\label{site-energy-contour}
  E_{\ell}(y) =-\frac{1}{2\pi i}\oint_{\mathscr{C}}\mathfrak{f}(z)\big[\rzz\big]_{\ell\ell}\dd z.
\end{eqnarray}
We can now calculate the first and second derivatives of $E_{\ell}(y)$ based on
\eqref{site-energy-contour} and the regularity assumption in {\bf H.loc}:
\begin{align}
\label{site-force}
\frac{\partial E_{\ell}(y)}{\partial [y(m)]_i} 
  &= \frac{1}{2\pi i} \oint_{\mathscr{C}} \mathfrak{f}(z) 
    \Big[\rzz\left[\mathcal{H}_{,m}\right]_i\rzz\Big]_{\ell\ell}\dd z, 
    \quad \text{and}  \\[1ex] 
  \frac{\partial^2 E_{\ell}(y)}{\partial [y(m)]_i\partial [y(n)]_j} 
  \notag
  &= \frac{1}{2\pi i}\oint_{\mathscr{C}}\mathfrak{f}(z) \bigg[
    \mathscr{R}_z \big[\mathcal{H}_{,mn}\big]_{ij} \mathscr{R}_z
    - \mathscr{R}_z \left[\mathcal{H}_{,m}\right]_i \mathscr{R}_z
    \big[\mathcal{H}_{,n} \big]_j  \mathscr{R}_z \\ \label{site-hessian} 
  & \hspace{13em} 
    - \mathscr{R}_z \left[\mathcal{H}_{,n} \right]_j \mathscr{R}_z
    \big[\mathcal{H}_{,m}\big]_i \mathscr{R}_z \bigg]_{\ell\ell} \dd z.
\end{align}
We also have higher order derivatives of the site energy for
$n\leq\mathfrak{n}$:
\begin{multline}\label{site-higher-derivative}
\frac{\partial^n E_{\ell}(y)}{\partial [y(m_1)]_{i_1}\cdots\partial [y(m_n)]_{i_n}} 
= -\frac{1}{2\pi i}\oint_{\mathscr{C}}\mathfrak{f}(z) \Bigg[\sum_{l=1}^{n}
\sum_{j_1+\cdots+ j_l=n}\sum_{\mathcal{P}_{m_1,\cdots,m_n}^{j_1,\cdots,j_l}} (-1)^l \\[1ex]
Q^{(l)}_z[y]\Big([\mathcal{H}_{,\mathcal{P}m_{1}\cdots\mathcal{P}m_{j_1}}(y)]_{\mathcal{P}i_1\cdots \mathcal{P}i_{j_1}},
\cdots,[\mathcal{H}_{,\mathcal{P}m_{n-j_l+1}\cdots\mathcal{P}m_{n}}(y)]_{\mathcal{P}i_{n-j_l+1}\cdots \mathcal{P}i_n}\Big) 
\Bigg]_{\ell\ell} \dd z, \quad
\end{multline}
where 
$Q^{(n)}_z:\left(\mathbb{R}^{N\times N}\right)^n\mapsto\mathbb{R}^{N\times N}$:
\begin{eqnarray}\label{def-Qk}
Q^{(n)}_z[y](\Theta_1,\cdots,\Theta_n) = \rz{y} \prod_{j=1}^n \Big(\Theta_j \rz{y}\Big)
\end{eqnarray}
is a well-defined linear map for any $z$ satisfying \eqref{dist-c}
and $\mathcal{P}_{m_1,\cdots,m_n}^{j_1,\cdots,j_l}$
is the multiset permutation of $m_1,\cdots,m_n$.

\subsection{Properties of the site energy}
\label{sec:properties_site_finite}
In order for $E_\ell$ to be a ``true'' site energy it must satisfy certain
properties: locality, permutation invariance and isometry invariance. We
establish these next.

First, we establish the locality of the site energy and its derivatives. We
remark that in this result it is important that we are keeping $\mu$
fixed. Admitting $y$-dependent $\mu$ would introduce a small amount of
non-locality in the site-energies, but it is reasonable to expect that this
vanishes in the thermodynamic limit.



\begin{lemma}[Locality] \label{lemma-locality-site}
  If \asL, \asHtb, {\bf H.loc}, {\bf F} are satisfied then, for
  $1\leq j\leq\mathfrak{n}$, there exist positive constants $C_j$ and $\eta_j$
  such that for any $\ell\in\Lambda_N$,
  \begin{eqnarray}\label{site-j-decay-finite}
    \left|\frac{\partial^j E_{\ell}(y)}{\partial [y(m_1)]_{i_1}\cdots\partial [y(m_j)]_{i_j}}\right|
    \leq C_j e^{-\eta_j\sum_{l=1}^j|y(\ell)-y(m_l)|} \qquad 1\leq i_1,\cdots,i_j\leq d.
  \end{eqnarray}
\end{lemma}

\def\rzz{\mathscr{R}_z}
\begin{proof}
  We will only give the explicit
  proofs for $j = 1, 2$, the cases $j > 2$ being analogous (but tedious).
  
  For $j=1$, we have from Lemma \ref{lemma-resolvant-decay} and the assumptions
  \asL, {\bf H.loc} that
  \begin{align} 
    \notag
    \Big[\rzz\left[\mathcal{H}_{,m}\right]_i\rzz \Big]_{\ell\ell} 
    &= 
      \sum_{\ell_1,\ell_2\in\Lambda_N} \big[\rzz \big]_{\ell\ell_1}
      \big( \left[\mathcal{H}_{,m}\right]_i \big)_{\ell_1\ell_2}
      \big[\rzz\big]_{\ell_2\ell} 
    \\[1ex]
    \notag
    &\leq
      c_{\rm r}^2 \bar{h}_1 
      \sum_{\ell_1,\ell_2\in\Lambda_N} 
      e^{ -\min\{\gamma_{\rm r},\gamma_1\}
      \big(r_{\ell\ell_1} + r_{\ell_1 m} + r_{m \ell_2} + r_{\ell_2 \ell} \big)} 
    \\
    \notag
    &=
      c_{\rm r}^2 \bar{h}_1  \bigg( \sum_{\ell_1\in\Lambda_N} 
      e^{ - \min\{\gamma_{\rm r},\gamma_1\} \big( r_{\ell\ell_1} 
      + r_{\ell_1 m} \big)} \bigg)^2 
    \\
    \label{eq:site-force-estimate-proof}
    &\leq
      Ce^{- \min\{\gamma_{\rm r},\gamma_1\}  |y(\ell)-y(m)|},
      \qquad\quad
  \end{align}
  where $C$ depends only on
  $\bar{h}_1, c_{\rm r}, \gamma_{\rm r}, \gamma_1, \mathfrak{m}$ and
  $d$. Together with \eqref{site-force} and {\bf F}, this leads to
  \begin{eqnarray}\label{site-force-decay-finite}
    \frac{\partial E_{\ell}(y)}{\partial [y(m)]_i} \leq C_1 e^{-\eta_1|y(\ell)-y(m)|} 
    \quad{\rm for}~1\leq i\leq d.
  \end{eqnarray}
	
  For $j=2$, we employ Lemma \ref{lemma-resolvant-decay} and \asL, {\bf H.loc}
  to estimate the three term arising in the expression \eqref{site-hessian} of
  the site Hessian, using similar computations as
  \eqref{eq:site-force-estimate-proof}:
  \begin{align*}
    & \hspace{-3em} 
      \Big[ \rzz \left[ \mathcal{H}_{,m} \right]_{i} \rzz 
      \left[ \mathcal{H}_{,n} \right]_{j} \rzz \Big]_{\ell\ell} 
    \\
    &\leq
      c_{\rm r}^3 \bar{h}_1^2 \sum_{\ell_1,\ell_2,\ell_3,\ell_4 \in \Lambda_N} 
      e^{ - \min\{\gamma_{\rm r},\gamma_1\}  \big( r_{\ell\ell_1} + r_{\ell_1 m}
      + r_{\ell_2 m} + r_{\ell_2 \ell_3} + r_{\ell_3 n} + r_{\ell_4 n}
      + r_{\ell_4 \ell} \big)}
    \\
    &\leq
      C e^{ - \frac12 \min\{\gamma_{\rm r},\gamma_1\} 
      ( |y(\ell)-y(m)|+|y(\ell)-y(n)| ) }; 
    \\[0.5em]
    & \hspace{-3em} 
      \Big[\rzz \left[ \mathcal{H}_{,n} \right]_{j} \rzz 
      \left[ \mathcal{H}_{,m} \right]_{i} \rzz \Big]_{\ell\ell} 
    \\
    &\leq 
      C e^{-\frac12 \min\{\gamma_{\rm r},\gamma_1\} (|y(\ell)-y(m)|+|y(\ell)-y(n)| )};
      \qquad \text{and}
    \\[0.5em]
    & \hspace{-3em} 
      \Big[ \rzz \left[ \mathcal{H}_{,mn} \right]_{ij} \rzz \Big]_{\ell\ell}
    \\
    &\leq
      c_{\rm r}^2 \bar{h}_2 \sum_{\ell_1,\ell_2\in\Lambda_N}
      e^{ - \min\{\gamma_{\rm r},\gamma_2\} ( r_{\ell \ell_1} + r_{\ell_1 m} + r_{\ell_1 n} 
      + r_{\ell_2 m} + r_{\ell_2 n} + r_{\ell_2 \ell} ) } 
    \\
    &\leq
      C e^{- \frac12 \min\{\gamma_{\rm r},\gamma_2\} (|y(\ell)-y(m)|+|y(\ell)-y(n)| )}.
  \end{align*}
  Inserting these three estimates into \eqref{site-hessian} together with {\bf
    F} yields the desired result,
  \begin{displaymath}\label{site-hessian-decay-finite}
    \frac{\partial^2 E_{\ell}(y)}{\partial [y(m)]_i\partial [y(n)]_j} \leq C_2 
    e^{-\eta_2\big(|y(\ell)-y(m)|+|y(\ell)-y(n)|\big)} \quad{\rm for}~1\leq i,j\leq d. 
    \qedhere
  \end{displaymath}	
\end{proof}

The next lemma summarises the consequences of {\bf H.sym}.

\begin{lemma}[Symmetries] \label{lemma-symmetry-site}
	Let $y\in\mathcal{V}^N_{\mathfrak{m}}$.         
	Assume that {\bf H.sym} is satisfied.
	\begin{itemize}
	\item[(i)] {\rm (Isometry invariance)}
	if $g:\R^d\rightarrow\R^d$ is an isometry, 
	then $E_{\ell}(y) = E_{\ell}(g(y))$;
	
	\item[(ii)] {\rm (Permutation invariance)}
	if $\mathcal{G}:\Lambda\rightarrow\Lambda$ is a permutation (relabelling)
	of $\Lambda$,
	then $E_{\ell}(y) = E_{\mathcal{G}^{-1}(\ell)}(y\circ\mathcal{G})$.	
	\end{itemize}
\end{lemma}

\begin{proof}
  \noindent {\it (i)}
  Let $y'=g(y)$.  Since $g$ is an isometry, we have that
  $y'\in\mathcal{V}^N_{\mathfrak{m}}$. The assumption {\bf H.sym (i)} implies
  $h_{mn}(y) = h_{mn}(y')$, which together with \eqref{site-energy-contour}
  yields
  \begin{eqnarray}\label{proof-5-1-2-f}
    E_{\ell}(y) = -\frac{1}{2\pi i}\oint_{\mathscr{C}}\mathfrak{f}(z) 
    \left[\rz{y}\right]_{\ell\ell} \dd z 
    = -\frac{1}{2\pi i}\oint_{\mathscr{C}}\mathfrak{f}(z)
    \left[\rz{y'}\right]_{\ell\ell} \dd z = E_{\ell}(y'). 
  \end{eqnarray}
  
  \vskip 0.2cm
  \noindent {\it (ii)}
  Let $y'=y\circ\mathcal{G}$.  The assumption {\bf H.sym (ii)} implies
  $h_{mn}(y) = h_{\mathcal{G}^{-1}(m)\mathcal{G}^{-1}(n)}(y')$, which together
  with \eqref{site-energy-contour} leads to
  $E_{\ell}(y) = E_{\mathcal{G}^{-1}(\ell)}(y')$.
\end{proof}


\section{Pointwise thermodynamic limit} 
\label{sec-thermodynamic-limits}
\setcounter{equation}{0}
Our aim in this section is to give a meaning to energy in the infinite body
limit (``thermodynamic limit''). The notion of site energy makes this relatively
straightforward: we will prove that, fixing a site $\ell$, and ``growing'' the
material around it to an infinite body yields a well-defined site energy
functional $E_\ell$ for an infinite body. Total energy in an infinite body is of
course ill-defined, but using the site energies it then becomes straightforward
to consider energy-differences; cf. \S~\ref{sec-crystal-defects}.

We need the following additional assumption, connecting the Hamiltonians for
growing index-sets, in our analysis.

\def\asHemb{{\bf H.emb}\xspace}

\begin{flushleft} \asHemb. \label{as:asHemb}
	Let $y^N:\Lambda_N\rightarrow\R^d$,
	$y:\Lambda_N\cup\{x'\}\rightarrow\R^d$ be two configurations satisfying {\bf
		L}, and $h^N_{\ell k}(y^N)$, $h_{\ell k}(y)$ be the corresponding
	Hamiltonian matrix elements of these two configurations satisfying {\bf H.loc}.
	If $y^N(\ell)=y(\ell)$ for any $\ell\in\Lambda_N$, 
	then for $0\leq j\leq\mathfrak{n}-1$,
	\begin{eqnarray}\label{h-locality}
	\frac{\partial^j h^N_{\ell k}(y^N)}{\partial [y(m_1)]_{i_1}\cdots\partial [y(m_j)]_{i_j}} 
	= \lim_{|y(x')|\rightarrow\infty}\frac{\partial^j h_{\ell k}(y)}{\partial [y(m_1)]_{i_1}
		\cdots\partial [y(m_j)]_{i_j}}\quad\forall~\ell,k\in\Lambda_N
	\end{eqnarray}
	with $m_1,\cdots,m_j\in\Lambda_N$ and $1\leq i_1,\cdots,i_j\leq d$.
\end{flushleft}

\begin{remark}
    Intuitively, \asHemb states that, if one atom $y(x')$ in the system is
    moved to infinity, then the Hamiltonian matrix elements for the remaining
    subsystem $\{ y(\ell) ~\lvert~ \ell \in \Lambda_N \}$ do not depend on $y(x')$
    anymore.
    
    At first glance, this appears to be a consequence of {\bf H.tb} and {\bf
      H.loc}. The reason we have to formulate it as a separate assumption is to
    make a connection between the Hamiltonians for $\Lambda_N$ and
    $\Lambda_N \cup \{x'\}$. More generally, in Lemma \ref{lemma-locality-h}, we
    obtain an analogous connection between the Hamiltonians for any two systems
    $\Lambda_N, \Lambda_M$ with $\Lambda_N \subset \Lambda_M$.
\end{remark}

Let $\Lambda$ be a countable index set (or, reference configuration), then
we consider {\em deformed configurations} belonging to the class
\begin{eqnarray}
\notag
\mathcal{V}_{\mathfrak{m}}(\Lambda) 
  &:=& \left\{ y : \Lambda \rightarrow \R^d,~~
    y|_{\Lambda_N}\subset\mathcal{V}_{\mathfrak{m}}^N 
    ~~{\rm for~any~finite~}\Lambda_N \subset\Lambda \right\} \\[2mm]
   	&=& \big\{ y : \Lambda \to \R^d, |y(\ell)-y(k)| 
   	\geq \frak{m} \text{ for all } \ell, k \in \Lambda \big\}. 
\end{eqnarray}
If $y \in \mathcal{V}_{\frak{m}}(\Lambda)$, then \asL is satisfied for any
finite subsystem $\Lambda_N \subset \Lambda$.  In the following, whenever we
assume {\bf H.tb}, {\bf H.loc} and \asHemb for infinite $\Lambda$, we mean that
they are satisfied for the Hamiltonian matrices of any finite subsystem
$\Lambda_N \subset \Lambda$.

For a bounded domain $\Omega\subset\R^d$, we shall denote the Hamiltonian,
resolvent and energy of the finite subsystem contained in $\Omega$,
respectively, by $\mathcal{H}^{\Omega}(y)$, $\rzo{\Omega}{y}$ and
$E^{\Omega}(y)$. For simplicity of notation, we drop the argument 
$(y)$ whenever convenient.

\begin{theorem}\label{theorem-thermodynamic-limit}
  Let $\Lambda$ be countable and $y \in \mathcal{V}_{\frak{m}}(\Lambda)$ be a deformation.
  Suppose the assumptions 
  {\bf F}, {\bf H.tb}, {\bf H.loc}, \asHemb and {\bf H.sym} 
  are satisfied for all finite subsystems 
  (with simultaneous choice of constants), then  
 \begin{itemize}
  	\item[(i)] {\rm (existence of the  thermodynamic limits)}
          for any $\ell \in \Lambda$ and for any sequence of convex and bounded
          sets $\Omega_R \supset B_R(y(\ell)), R > 0$ the limit
    \begin{eqnarray}\label{site-limit-new}
      E_{\ell}(y) := \lim_{R\rightarrow\infty}E^{\Omega_R}_{\ell}(y)
    \end{eqnarray}
    exists and is independent of the choice of sets $\Omega_R$;
  \item[(ii)] {\rm (regularity and locality of the limits)} the limits
    $E_{\ell}(y)$ possess $j$th order partial derivatives with
    $1 \leq j \leq \mathfrak{n}-1$, and it holds that
	\begin{eqnarray}\label{site-locality-tdl}
	\left|\frac{\partial^j E_{\ell}(y)}{\partial [y(m_1)]_{i_1}\cdots
		\partial [y(m_j)]_{i_j}}\right|\leq C_j e^{-\eta_j\sum_{l=1}^j|y(\ell)-y(m_l)|} 
	\quad 1\leq i_1,\cdots,i_j\leq d,
	\end{eqnarray}
	where the constants $C_j$ and $\eta_j$ are same to those in Lemma
        \ref{lemma-locality-site};

	\item[(iii)] {\rm (isometry invariance)}
	if $g:\R^d\rightarrow\R^d$ is an isometry, 
	then $E_{\ell}(y) = E_{\ell}(g(y))$;
	
	\item[(iv)] (permutation invariance)
	if $\mathcal{G}:\Lambda\rightarrow\Lambda$ is a permutation (relabelling)
	of $\Lambda$,
	then $E_{\ell}(y) = E_{\mathcal{G}^{-1}(\ell)}(y\circ\mathcal{G})$.
 \end{itemize}
\end{theorem}

Before we prove Theorem \ref{theorem-thermodynamic-limit} we establish an
  extension of \asHemb.

\begin{lemma}\label{lemma-locality-h}
	Let $\Lambda_M \supsetneq \Lambda_N$, $y^M:\Lambda_M\rightarrow\R^d$ and
	$y^N:\Lambda_N\rightarrow\R^d$ be two configurations satisfying \asL.
	Assume that there exists a convex set $\Omega\subset\R^d$ such that
	\begin{eqnarray*}
		y^M(\ell)=y^N(\ell)\in\Omega\quad\forall~\ell\in\Lambda_N
		\qquad{\rm and}\qquad
		y^M(\ell)\in\Omega^{\rm c}\quad\forall~\ell\in\Lambda_M\backslash\Lambda_N,
	\end{eqnarray*}
	where $\Omega^{\rm c}$ is the complement of $\Omega$ in $\R^d$.  
	If {\bf F},	{\bf H.loc, H.ext} are satisfied then, for $0\leq j\leq\mathfrak{n}-1$, 
	there exist positive constants $c_j, \kappa_j$, 
	which do not depend on $M$, $N$ and $\Omega$, such that
	\begin{eqnarray}\label{h-locality-0}    
	|h_{\ell k}^M(y^M)-h_{\ell k}^N(y^N)| 
	\leq c_0 \exp\Big(- \kappa_0 \Big( \dis{y^N(\ell)}{\Omega^{\rm c}}
	+  \dis{y^N(k)}{\Omega^{\rm c}} \Big) \Big), 
	\end{eqnarray}
	and
	\begin{eqnarray}
	\notag
	\left| 
	\frac{\partial^j h^M_{\ell k}(y^M)}{\partial [y(m_1)]_{i_1} \cdots \partial [y(m_j)]_{i_j}}
	- \frac{\partial^j h^N_{\ell k}(y^N)}{\partial [y(m_1)]_{i_1}\cdots \partial [y(m_j)]_{i_j}} 
	\right| \hspace{-25em} & 
	\\[0.2em]
	\label{h-locality-j}
	&\leq 
	c_j \exp\Big( -\kappa_j \Big(\dis{y^N(\ell)}{\Omega^{\rm c}} 
	+ \dis{y^N(k)}{\Omega^{\rm c}} \\
	& \notag
	\hspace{14em} + \sum_{s=1}^j(|y^N(\ell)-y^N(m_s)|+|y^N(k)-y^N(m_s)|)
	\Big) \Big)
	\end{eqnarray}
	for any $\ell,k,m_s\in\Lambda_N$ and $1\leq i_1,\dots,i_j\leq d$.
\end{lemma}

\begin{proof}
  We first prove the case $j=0$, i.e., \eqref{h-locality-0}.
  For each $\ell\in \Lambda_M\backslash\Lambda_N$, we can find a normalized
  vector $\nu_{\ell}$ such that
  $y^M(\ell)-\nu_{\ell}\cdot\dis{y^M(\ell)}{\Omega}\in \overline{\Omega}$.
  Let $\{R^{\mu}_{\ell}\}_{\mu\in\N}$ be a sequence for each
  $\ell\in \Lambda_M\backslash\Lambda_N$, such that
  $R^{\mu}_{\ell}\rightarrow\infty$ as $\mu\rightarrow\infty$, then we define
  \begin{eqnarray*}
    \tilde{y}^{\mu}(\ell):=\left\{
    \begin{array}{ll}
      y^N(\ell) & {\rm if}~\ell\in\Lambda_N; \\[1ex]
      y^M(\ell) + {\nu}_{\ell}\cdot R^{\mu}_{\ell} & {\rm if}~\ell\in\Lambda_M\backslash\Lambda_N.
    \end{array}\right.
  \end{eqnarray*}
  Using \asHemb with $j=0$ and an elementary argument we can inductively
    choose $R_\ell^\mu$, such that
  $R^{\mu}:=\min_{\ell\in \Lambda_M\backslash\Lambda_N}R^{\mu}_{\ell}
  \rightarrow\infty$ as $\mu\rightarrow\infty$ and
    \begin{eqnarray}\label{proof-1-2-1}
      h_{\ell k}^N(y^N)=\lim_{\mu\rightarrow\infty}h_{\ell k}^M(\tilde{y}^{\mu}).
    \end{eqnarray}
    
    Note that the convexity of $\Omega$ implies
    $$
    |y^M(\ell)-y^M(n) + c\nu_{n}| \geq \frac{1}{2}\Big( |y^M(\ell)-y^M(n)| + c\Big)
    $$
    for any $n\in M\backslash N$ and $c>0$.  Using the assumptions \asL and {\bf
      H.loc} with $j=1$, we have
	\begin{eqnarray*}
		&&   \hspace{-2.5em}
		\big|h_{\ell k}^M(\tilde{y}^{\mu})-h_{\ell k}^M(y^M)\big|
		~=~ \left|\int_0^1
		\sum_{n \in \Lambda_M \setminus \Lambda_N} 
		\frac{\partial h_{\ell k}^M}{\partial y(n)}
		\Big( (1-t) y^M + t \tilde{y}^{\mu} \Big) \cdot 
		\big( \tilde{y}^{\mu}(n) - y^M(n) \big)\dd t \right| \\[1ex]
		&\leq & C\bar{h}_1 \int_0^1  \sum_{n\in\Lambda_M\backslash\Lambda_N} 
		e^{-\gamma_1 \big(|y^M(\ell)-y^M(n)+tR^{\mu}\nu_{n}| 
			+ |y^M(k)-y^M(n)+tR^{\mu}\nu_{n}|\big)} \cdot R^{\mu} \dd t
		\\[1ex]
		&\leq & C\bar{h}_1  \sum_{n\in\Lambda_M\backslash\Lambda_N}  
		e^{-\frac{\gamma_1}{2} \big(|y^M(\ell)-y^M(n)|+|y^M(k)-y^M(n)|\big)} 
		\cdot  \int_0^1 e^{-\gamma_1 tR^{\mu}}R^{\mu} \dd t
		\\[1ex]
		&\leq & c_0 e^{-\kappa_0\big(\dis{y^N(\ell)}{\Omega^{\rm c}} 
			+ \dis{y^N(k)}{\Omega^{\rm c}}\big)} \cdot  \big(1-e^{-\gamma_1 R^{\mu}}\big) ,
	\end{eqnarray*}
	where $\kappa_0=\gamma_1/2-\varepsilon$ with any 
	$\varepsilon\in(0,\gamma_1/2)$, and $c_0$ is a constant
	depending only on $\varepsilon$, $\bar{h}_1$, $\gamma_1$,
	$\mathfrak{m}$ and $d$. Note that neither $\kappa_0$ nor $c_0$ depends
	on the system size $M$, $N$ and $\Omega$.
	Note that
	\begin{eqnarray*}
		\big|h_{\ell k}^M(y^M)-h_{\ell k}^N(y^N)\big| 
		= \lim_{\mu\rightarrow\infty}\big|h_{\ell k}^M(y^M)-h_{\ell k}^M(\tilde{y}^{\mu})\big| 
		\leq c_0 e^{-\kappa_0\big(\dis{y^N(\ell)}{\Omega^{\rm c}} + \dis{y^N(k)}{\Omega^{\rm c}}\big)},
	\end{eqnarray*}
	which completes the proof of the case $j=0$.
	
	With the same arguments, we can prove \eqref{h-locality-j} for  
	$1\leq j\leq\mathfrak{n}-1$ by using the assumptions {\bf H.loc} 
	with index $j+1$ and \asHemb with index $j$.
\end{proof}

\begin{proof} [Proof of Theorem \ref{theorem-thermodynamic-limit}]	
	{\it (i)}	
	Without loss of generality, we can assume that the upper bound of the spectrum
	$\overline{\sigma}< 0 $  
	(one can always shift the eigenvalues if this is not satisfied) 
	and the contour $\mathscr{C}$ is chosen such that it includes 0 and
	\begin{eqnarray}\label{dist-c-0}
	\min\left\{ {\rm dist}\big(\mathscr{C},\sigma(\mathcal{H}(y))\big),
	{\rm dist}\big(\mathscr{C},\mathfrak{s}(\mathfrak{f})\big),
	{\rm dist}\big(\mathscr{C},\{0\}\big) \right\} \geq \mathfrak{d}.
	\end{eqnarray}
	
  Let $\ell \in \Lambda$ and $B_R(y(\ell)) \subset \Omega_R \subsetneq \Omega'$, then
  we define
  \begin{eqnarray}\label{ham-extended}
    \Big[\widetilde{\mathcal{H}}^{\Omega_R}(y) \Big]_{km} 
    := \left\{ 
    \begin{array}{ll}
      \big[\mathcal{H}^{\Omega_R}(y)\big]_{km} 
      & \text{if } y(k), y(m) \in \Omega_R; \\[1ex]
      0 & \text{otherwise for } y(k), y(m) \in \Omega'.
    \end{array}\right.
  \end{eqnarray}
  Note that \eqref{dist-c-0} implies that the condition \eqref{dist-c} is 
  satisfied for the Hamiltonian  $\widetilde{\mathcal{H}}^{\Omega_R}$ with
  the contour $\mathscr{C}$. Moreover, 
  the resolvent $\rzotilR = (\widetilde{\mathcal{H}}^{\Omega_R}(y)-zI)^{-1}$
  is well-defined for any $z\in\mathscr{C}$ and satisfies the estimate in
  \eqref{resolvant-decay}.
  Under the assumption {\bf F}, we can observe that the band energy and site energies
  of the Hamiltonian $\mathcal{H}^{\Omega_R}(y)$ and
  $\widetilde{\mathcal{H}}^{\Omega_R}(y)$ are the same.

  We have from \eqref{site-energy-contour}, \eqref{ham-extended}, \asL, 
  {\bf H.tb}, Lemma \ref{lemma-locality-h} and Lemma \ref{lemma-resolvant-decay} 
  that 
  \begin{align} \nonumber
    & \hspace{-3em} 
      E^{\Omega'}_{\ell}(y)-E^{\Omega_R}_{\ell}(y) 
      ~=~ -\frac{1}{2\pi i} \oint_{\mathscr{C}} \mathfrak{f}(z) \Big[ 
      \mathscr{R}_z^{\Omega'} - \widetilde{\mathscr{R}}_z^{\Omega_R}
      \Big]_{\ell\ell} \dd z  \\[1ex] \nonumber
    &\leq 
      C \sum_{y(\ell_1),y(\ell_2) \in \Omega'} 
      \big[ \mathscr{R}_z^{\Omega'} \big]_{\ell\ell_1} \big[ \mathcal{H}^{\Omega'} -
      \widetilde{\mathcal{H}}^{\Omega_R} \big]_{\ell_1\ell_2}
      [\rzoR ]_{\ell_2\ell}
    \\[1ex] \nonumber
    &\leq 
      C \sum_{y(\ell_1),y(\ell_2) \in \Omega'} e^{-\gamma_{\bf r}r_{\ell\ell_1}}
      e^{-\kappa_0\big(\dis{y(\ell_1)}{\Omega' \backslash \Omega_R} 
      + \dis{y(\ell_2)}{\Omega'\backslash \Omega_R}\big)} 
      e^{-\gamma_{\bf r}r_{\ell\ell_2}} \\[1ex]    
    &\leq 
      C \left( \sum_{y(\ell_1)\in\Omega'} e^{-\gamma_{\bf r} r_{\ell\ell_1} - 
      \kappa_0 \dis{y(\ell_1)}{\Omega' \backslash \Omega_R} } \right)^2
      ~\leq~ C e^{-\min\{\gamma_{\bf r},\kappa_0\}R},
      \label{proof-2-3-1}
  \end{align}
  where the last constant $C$ depends only on $\bar{h}_0$, $c_{\rm r}$, $c_0$,
  $\gamma_{\bf r}$, $\kappa_0$, $\mathfrak{m}$ and $d$, but is independent of
  $y$ or $R$. 
  Since \eqref{proof-2-3-1} holds for any $\Omega'\supsetneq B_R(y(\ell))$
  it follows that $\left\{E^{\Omega_R}_{\ell}(y)\right\}_{R\in\N}$ is a Cauchy
  sequence. The uniqueness of the limit is also an immediate consequence of
  the fact that $\Omega'$ was arbitrary. This completes the proof of {\it (i)}.

  \vskip 0.2cm
  \noindent {\it (ii)}  
  {\it Case $j = 1$:} 
  
  For $\ell,m\in\Lambda$, we take $R>2|y(m)-y(\ell)|$, and
  then adopt the notation in the proof of {\it (i)}. 
  %
  %
  With the expression of \eqref{site-force}, we obtain by a direct calculation
  that
  \begin{multline}\label{proof-2-5-1}
  \frac{\partial E^{\Omega'}_{\ell}(y)}{\partial [y(m)]_i} - 
  \frac{\partial E^{\Omega_R}_{\ell}(y)}{\partial [y(m)]_i} 
  =
  \frac{1}{2\pi i} \oint_{\mathscr{C}} \mathfrak{f}(z) \Bigg[ 
  \restilde \left(\widetilde{\mathcal{H}}^{\Omega_R}-\mathcal{H}^{\Omega'}\right) \resprime
  \big[\widetilde{\mathcal{H}}^{\Omega_R}_{,m}\big]_i \resprime	+ 
  \\[1ex]  
  \restilde \big[ \mathcal{H}^{\Omega'}_{,m} \big]_i \resprime \left( 
  \widetilde{\mathcal{H}}^{\Omega_R} - \mathcal{H}^{\Omega'} \right) \restilde  
  + \restilde \left( \big[\mathcal{H}^{\Omega'}_{,m}\big]_i -
  \big[ \widetilde{\mathcal{H}}^{\Omega_R}_{,m} \big]_i
  \right) \widetilde{\mathscr{R}}^{\Omega_R}_z \Bigg]_{\ell\ell} \dd z.
  \end{multline}
  Using \asL, \asHtb, Lemma \ref{lemma-locality-h} and
  Lemma~\ref{lemma-resolvant-decay}, we can obtain from a similar argument as in
  \eqref{proof-2-3-1} that
  \begin{eqnarray}\label{proof-2-6-1}
  \left|\frac{\partial E^{\Omega'}_{\ell}(y)}{\partial [y(m)]_i} 
  - \frac{\partial E^{\Omega_R}_{\ell}(y)}{\partial [y(m)]_i}\right| 
  \leq Ce^{-\min\{\gamma_{\bf r},\kappa_0,\kappa_1\}\big(R+|y(\ell)-y(m)|\big)/2},
  \end{eqnarray}
  where the constant $C$ depends only on $\bar{h}_0, c_{\rm r}$, $c_0$, $c_1$,
  $\gamma_{\bf r}$, $\kappa_0$, $\kappa_1$, $\mathfrak{m}$ and $d$.
  Note that the estimate in \eqref{proof-2-6-1} can be bounded by
  $Ce^{-\min\{\gamma_{\bf r},\kappa_0,\kappa_1\}R/2}$, which does not depend on
  $y$. Therefore, we have that
  $\left\{\partial E^R_{\ell}(y)/\partial [y(m)]_i\right\}_{R\in\N}$ converge
  uniformly to some limit, which together with {\it (i)} implies that 
  $E_{\ell}(y)$ is differentiable with respect to $y(m)$ and the derivative 
  is given by
  \begin{eqnarray}\label{site-1st}
  \frac{\partial E_{\ell}(y)}{\partial [y(m)]_i}=\lim_{R\rightarrow\infty}
  \frac{\partial E^{B_R(y(\ell))}_{\ell}(y)}{\partial [y(m)]_i} 
  \quad~{\rm for}~1\leq i\leq d.
  \end{eqnarray}
  
  {\it Case $j > 1$:} 
  
  For the second order derivatives, we can obtain from the expression 
  \eqref{site-hessian} and a tedious calculation that
  \begin{align*}
  & \hspace{-3em} 
  \frac{\partial^2 E^{\Omega'}_{\ell}(y)}{\partial [y(m)]_i \partial [y(n)]_j} 
  - \frac{\partial^2 E^{\Omega_R}_{\ell}(y)}{\partial [y(m)]_i \partial [y(n)]_j} 
  \\[1ex]   
  =~ & \frac{1}{2\pi i}\oint_{\mathscr{C}}\mathfrak{f}(z) \Bigg[  \Bigg.
  \restilde \left(\mathcal{H}^{\Omega'}-\widetilde{\mathcal{H}}^{\Omega_R}\right)  
  \resprime \big[\widetilde{\mathcal{H}}^{\Omega_R}_{,m}\big]_{i} \restilde 
  \big[\mathcal{H}^{\Omega'}_{,n}\big]_{j} \restilde 
  \\[1ex]  
  & +
  \resprime \left(\big[\widetilde{\mathcal{H}}_{,m}^{\Omega_R}\big]_i
   - \big[\mathcal{H}_{,m}^{\Omega'}\big]_i\right) \restilde 
   \big[\widetilde{\mathcal{H}}^{\Omega_R}_{,n}\big]_{j} \resprime
  \\[1ex]
  & + 
  \restilde \left(\widetilde{\mathcal{H}}^{\Omega_R}-\mathcal{H}^{\Omega'}\right)
  \resprime \big[\widetilde{\mathcal{H}}^{\Omega_R}_{,mn}\big]_{ij} \resprime
  \\[1ex]  
  &  +
  \restilde \left(\big[\mathcal{H}^{\Omega'}_{,mn}\big]_{ij} 
  - \big[\widetilde{\mathcal{H}}^{\Omega_R}_{,mn}\big]_{ij} \right) \restilde
  +
  \cdots \Bigg. \Bigg]_{\ell\ell} \dd z. \qquad
  \end{align*}
  For readability, we have omitted eight other terms in the square bracket, 
  which have the same structure as the listed four terms.
  By estimating each term using the same arguments in \eqref{proof-2-3-1}, we can obtain
  \begin{eqnarray}\label{proof-2-6-2}
  \left|\frac{\partial^2 E^{\Omega'}_{\ell}(y)}{\partial [y(m)]_i \partial [y(n)]_j} 
  - \frac{\partial^2 E^{\Omega_R}_{\ell}(y)}{\partial [y(m)]_i \partial [y(n)]_j}\right| 
  \leq Ce^{-\min\{\gamma_{\bf r},\kappa_0,\kappa_1,\kappa_2\}\big(R+r_{\ell m}+r_{\ell n}\big)/4},~~
  \end{eqnarray}
  which leads to the existence of 
  $\partial^2 E_{\ell}(y) / \partial [y(m)]_i \partial [y(n)]_j$ 
  and
  \begin{eqnarray}\label{site-2nd}
  \frac{\partial^2 E_{\ell}(y)}{\partial [y(m)]_i \partial [y(n)]_j}
  =\lim_{R\rightarrow\infty}
  \frac{\partial^2 E^{\Omega_R}_{\ell}(y)}{\partial [y(m)]_i \partial [y(n)]_j} 
  \quad~{\rm for}~1\leq i,j\leq d.
  \end{eqnarray}
  
  For $2<j\leq\mathfrak{n}-1$, we can obtain by similar arguments that	\begin{eqnarray}\label{site-jth}
  \frac{\partial^j E_{\ell}(y)}{\partial [y(m_1)]_{i_1}\cdots\partial [y(m_j)]_{i_j}} 
  = \lim_{R\rightarrow\infty}\frac{\partial^j E^{\Omega_R}_{\ell}(y)}{\partial [y(m_1)]_{i_1}
  	\cdots\partial [y(m_j)]_{i_j}} \quad 1\leq i_1,\cdots,i_j\leq d ~~
  \end{eqnarray}  
  for any sequence $\Omega_R$ satisfying the conditions of {\it (i)}.
  
  The locality of $E_\ell$ in {\it (ii)} is now an immediate consequence of Lemma
  \ref{lemma-locality-site} and \eqref{site-1st}, \eqref{site-2nd}, \eqref{site-jth}.

  \vskip 0.2cm
  \noindent {\it (iii)}
  Let $y'=g(y)$.
  Since $g$ is an isometry, we have that $g\big(B_R(y(\ell))\big)=B_R(y'(\ell))$
  for any $R>0$.  
  We can obtain from {\bf H.sym (i)} and Lemma \ref{lemma-symmetry-site}(i) that
  \begin{eqnarray}\label{proof-5-1-7}
  E_{\ell}^{B_R(y(\ell))}(y) = E_{\ell}^{B_R(y'(\ell))}(y').
  \end{eqnarray}
  Taking the limit $R\rightarrow\infty$ of \eqref{proof-5-1-7} and {\it (i)} yield $E_{\ell}(y)=E_k(y')$.
  
  \vskip 0.2cm
  \noindent {\it (iv)}
  Similar to the proof of {\it (iii)}, {\it (iv)} is a consequence of
  Lemma \ref{lemma-symmetry-site}(ii) and {\it (i)}.
%
\end{proof}

Theorem \ref{theorem-thermodynamic-limit} states the existence of the
thermodynamic limits of the site energies, as well as the regularity, locality
and isometry/permutation invariance of the limits.  In the following we shall
always denote this limiting site energy by $E_\ell$.

\begin{remark}
  We have only considered the band energy of the system so far.  The repulsive
  energy $E^{\rm rep}$ can be incorporated into our analysis without difficulty.

  Using the expression \eqref{e_rep_l} and the assumption {\bf U}, it is easy
  to justify the thermodynamic limit of the repulsive site energy
  $E_{\ell}^{\rm rep}$, as well as its regularity and locality as those in
  Theorem \ref{theorem-thermodynamic-limit}.
  Moreover, the symmetry results in Theorem \ref{theorem-thermodynamic-limit}
  are also clearly satisfied with the expression \eqref{e_rep_l} for
  $E_{\ell}^{\rm rep}$.
  Therefore, all we have to do is to take the total site energy
  $$
  E_{\ell}^{\rm tot} = E_{\ell} + E_{\ell}^{\rm rep}
  $$
  and then use the existing results for $E_{\ell}$.  For convenience and
  readability, we still work with the site band energy $E_\ell$ and continue to
  ignore the repulsive component.
\end{remark}


\section{Applications} 
\label{sec-crystal-defects}
\setcounter{equation}{0}

\def\Rcore{R_{\rm def}}
\def\Adm{{\rm Adm}}
\def\E{\mathcal{E}}
\def\L{\Lambda}
\def\UsH{\dot{\mathscr{W}}^{1,2}}
\def\Usz{\dot{\mathscr{W}^{\rm c}}}

\subsection{Tight-binding model for point defects}
\label{sec:tb model for point defects}
As alluded to in the Introduction, our primary aim in understanding the locality
of the TB model is the construction and rigorous analysis of QM/MM hybrid
schemes for crystalline defects, along the lines of \cite{ehrlacher13}. The next
step towards this end is a rigorous definition of a variational problem that is
to be solved. Since we have shown in Theorem \ref{theorem-thermodynamic-limit}
that the total TB energy can be split into exponentially localised site
energies, this is a relatively straightforward generalisation of the analysis in
\cite{ehrlacher13}, which considers MM site energies with bounded interaction
radius. 

Here, we only summarize the results, with an eye to the application we present
in \S~\ref{sec:appl:truncation}. For simplicity we restrict ourselves to point
defects only. Complete proofs and generalisations to general dislocation
structures are given in \cite{chenpre_vardef}.

We call an index set $\L$ a {\em point defect reference configuration} if 

\def\asD{{\bf D}\xspace}

\begin{flushleft}
  \asD.   \label{as:asD}
  $\exists ~\Rcore>0, A \in {\rm SL}(d)$ such that
  $\Lambda\backslash B_{\Rcore} = (A\mathbb{Z}^d)\backslash B_{\Rcore}$ and
  $\L \cap B_{\Rcore}$ is finite.
\end{flushleft}

While, in previous sections, we have worked with {\em deformations} $y$ where
$y(\ell)$ denotes the deformed position of an atom indexed by $\ell$, it is now
more convenient to work with displacements $u : \L \to \R^d$,
$u(\ell) = y(\ell) - \ell$. For displacements $u$ we define the {\em
  energy-difference functional}
\begin{equation}
  \label{eq:ediff-before-renormalisation}
  \E(u) := \sum_{\ell \in \L} \big[ E_\ell(x+u) - E_\ell(x) \big],
\end{equation}
where $x$ denotes the identity map $x : \L \to \R^d, x(\ell) = \ell$. Due to the
exponential localisation of $E_\ell$ this series converges absolutely if $u$ has
compact support, i.e., for $u \in \Usz$ with
\begin{displaymath}
  \Usz := \big\{ u : \L \to \R^d, \exists R > 0 \text{ s.t. } u = {\rm const}
  \text{ in } \L \setminus B_R \big \};
\end{displaymath}
cf. Theorem \ref{th:appl:ediff}(i).

Next, still following \cite{ehrlacher13}, we extend the definition of $\E$ to a
natural energy space. For
$\ell \in \L, \rho \in \L - \ell := \{ m-\ell, m \in \L \setminus \{\ell\}\}$ we
define $D_\rho u(\ell) := u(\ell+\rho)-u(\ell)$, and moreover,
\begin{displaymath}
  Du(\ell) := \big( D_\rho u(\ell) \big)_{\rho \in \L-\ell}.
\end{displaymath}
We think of $Du(\ell) \in (\R^d)^{\L-\ell}$ as an (infinite) finite-difference
stencil. For any such stencil $Du(\ell)$ and $\gamma > 0$ we define the norm
\begin{displaymath}
  \big|Du(\ell)\big|_\gamma := \bigg( \sum_{\rho \in \L-\ell} e^{-2\gamma|\rho|} 
  \big|D_\rho u(\ell)\big|^2 \bigg)^{1/2},
\end{displaymath}
which gives rise to an associated semi-norm on displacements,
\begin{displaymath}
  \| Du \|_{\ell^2_\gamma} := \bigg( \sum_{\ell \in \L} |Du(\ell)|_\gamma^2 \bigg)^{1/2}.
\end{displaymath}
All (semi-)norms $\|\cdot\|_{\ell^2_\gamma}, \gamma > 0,$ are equivalent.  With
these definitions we can now define the function space, which encodes the
far-field boundary condition for displacements,
\begin{displaymath}
  \UsH := \big\{ u : \L \to \R^d, \| Du \|_{\ell^2_\gamma} < \infty \big\}.
\end{displaymath}
We remark that $\Usz$ is dense in $\UsH$ \cite{chenpre_vardef}.

In addition to the decay at infinity imposed by the condition $u \in \UsH$ we
also require a variant of \asL, stating that atoms do not collide. Thus, our set
of {\em admissible displacements} becomes
\begin{displaymath}
  \Adm_{\frak{m}} := \big\{ u \in \UsH, 
  |\ell+u(\ell)-m-u(m)| > \frak{m} |\ell-m| ~\forall~  \ell, m \in \L
  \big\},
\end{displaymath}
where $\frak{m}$ is an arbitrary positive number. Again, we can observe that
$\Usz \cap \Adm_{\frak{m}}$ is dense in $\Adm_{\frak{m}}$
\cite{chenpre_vardef}. We remark also that, due to the decay imposed by the
condition $u \in \UsH$, if $u \in \Adm_0$ then $u \in \Adm_{\frak{m}}$ for some
$\frak{m} > 0$. 

We can now state the main result concerning the energy-difference functional
$\E$. The proof is an extension of \cite[Lemma 2.1]{ehrlacher13} and will be
detailed in \cite{chenpre_vardef}. The main new ingredient in this extension, as
well as in Theorem \ref{th:appl:regularity} below, is quantifying how rapidly
the site energies approach those of a homogeneous crystal (without defect).

\begin{theorem}
  \label{th:appl:ediff}
  Suppose that \asD is satisfied, as well as {\bf F}, {\bf H.tb}, {\bf H.loc},
  \asHemb and {\bf H.sym} for all finite subsystems with simultaneous choice of
  constants.
  
  (i) $\E : \Usz \cap \Adm_0 \to \R$ is well-defined by
  \eqref{eq:ediff-before-renormalisation}, in the sense that the series
  converges absolutely.
  
  (ii) $\E : \Usz \cap \Adm_0 \to \R$ is continuous with respect to the
  $\|\cdot\|_{\ell^2_\gamma}$ semi-norm; hence, there exists a unique continuous
  extension to $\Adm_0$, which we still denote by $\E$.
  
  (iii) $\E \in C^{\mathfrak{n}-1}(\Adm_0)$ in the sense of Fr\'echet.
\end{theorem}

In view of Theorem \ref{th:appl:ediff} the following variational problem is
well-defined:
\begin{equation}
  \label{eq:appl:variational-problem}
  \bar u \in \arg\min \big\{ \E(u), u \in \Adm_0 \big\},
\end{equation}
where ``$\arg\min$'' is understood in the sense of local minimality. We are not
concerned with existence or uniqueness of minimizers but only their
structure. This is discussed in the next result, which is an extension of
\cite[Thm. 2.3]{ehrlacher13} (see \cite{chenpre_vardef} for the complete proof).

\begin{theorem}
  \label{th:appl:regularity}
  Suppose that \asD is satisfied, as well as {\bf F}, {\bf H.tb}, {\bf H.loc},
  \asHemb and {\bf H.sym} for all finite subsystems with simultaneous choice of
  constants.  If $\bar u \in \Adm_0$ is a strongly stable solution to
  \eqref{eq:appl:variational-problem}, that is,
  \begin{equation}
    \label{eq:appl:strong-stab}
    \exists~ \bar{c} > 0 \text{ s.t. } 
    \big\< \delta^2 \E(\bar u) v, v \big\> \geq \bar{c} \| Dv \|_{\ell^2_\gamma}^2
    \qquad \forall v \in \Usz,
  \end{equation}
  then there exists a constant $C > 0$ such that $\bar u$ satisfies the decay
  \begin{align*}
    |D\bar{u}(\ell)|_\gamma \leq C (1+|\ell|)^{-d} \qquad \forall \ell \in \L.
  \end{align*}
\end{theorem}

\begin{remark}
  (i) The condition \eqref{eq:appl:strong-stab} is stronger than actually
  required. Indeed, it suffices that $\bar u$ is a critical point and that
  strong stability is satisfied only in the far-field;
  cf. \cite[Eq. (2.7)]{ehrlacher13}.

  (ii) Higher-order decay estimates can be proven for higher-order gradients.
  For example, when $\mathfrak{n}\geq 5$, then
  $|D_{\rho_1}D_{\rho_2} \bar u(\ell)| \leq C |\ell|^{-d-1}$ for $|\ell|$
  sufficiently large; see \cite{ehrlacher13,chenpre_vardef} for more
  details. These estimates will be useful in our companion papers
  \cite{chenpreE, chenpreF} for the construction of highly accurate MM
  potentials, but are not required in the present work.
\end{remark}

\subsection{Convergence of a numerical scheme}
\label{sec:appl:truncation}
As a reference scheme to compare our QM/MM schemes against, and also as an
elementary demonstration of the usefulness of the locality results and of the
framework of \S~\ref{sec:tb model for point defects} we present an approximation
error analysis for a basic truncation scheme.

To construct the scheme we first prescribe a radius $R > 0$ and restrict the set
of admissible displacements to
\begin{displaymath}
  \Adm_{0}(R) := \big\{ u \in \Adm_0, u = 0 \text{ in } \L \setminus B_R \big\}.
\end{displaymath}
The pure Galerkin scheme $\bar u_R \in \arg\min \{ \E(u), u \in \Adm_0(R) \}$
is analyzed in \cite{ehrlacher13} and the convergence rate $\| D\bar u - D \bar
u_R \|_{\ell^2_\gamma} \lesssim R^{-d/2}$ is proven. 

\def\Rbuf{R^{\rm buf}}

In our case, the energy-difference $\E(u)$ is not computable for
$u \in \Adm_0(R)$ due to the infinite interaction radius of the TB
model. However, we can exploit the exponential localisation to truncate it. To
that end, we let $\Rbuf$ be a buffer region width
(cf. Theorem~\ref{th:appl:min-approx} and Remark \ref{rem:buffer radius}),
$\L_R := \L \cap B_{R+\Rbuf}$ and for any $v : \L \to \R^d$ define
$v^R : \L_R \to \R^d$ satisfying $v^R = v$ on $\L_R$. 
Then, for
  $u \in \Adm_0$, we define the truncated energy-difference functional
\begin{displaymath}
  \E_R(u) := E^{\L_R}\big([x+u]^R\big) - E^{\L_R}\big(x^R\big).
\end{displaymath}
Clearly, $\E_R$ is well-defined and $\E_R \in C^{\mathfrak{n}-1}(\Adm_0)$ in the
sense of Fr\'echet. To formulate the computational scheme, $\E_R$ need only be
defined for $u \in \Adm_0(R)$, but for the analysis it will be convenient to
define it for all $u \in \Adm_0$.

The computational scheme is now given by
\begin{equation}
  \label{eq:appl:min-approx}
  \bar u_R \in \arg\min \big\{ \E_R(u), u \in \Adm_0(R) \big\}.
\end{equation}

\begin{theorem}
  \label{th:appl:min-approx}
  Suppose that \asD is satisfied, as well as {\bf F}, {\bf H.tb}, {\bf H.loc},
  \asHemb and {\bf H.sym} for all finite subsystems with simultaneous choice of
  constants.

  If $\bar u$ is a strongly stable solution to
  \eqref{eq:appl:variational-problem}, then there are constants
  $C, R_0, c_{\rm buf}$ such that, for $R \geq R_0$ and
  $\Rbuf \geq c_{\rm buf} \log(R)$, there exists a strongly stable solution
  $\bar u_R$ to \eqref{eq:appl:min-approx} satisfying
  \begin{align} \label{estimate-Du}
    \big\| D\bar u - D\bar u_R \big\|_{\ell^2_\gamma} &\leq C R^{-d/2}, 
    \qquad \text{and} \\\label{estimate-energy}
    \big| \E(\bar u) - \E_R(\bar u_R) \big| &\leq C R^{-d}.
  \end{align}
\end{theorem}
\begin{proof}
  We closely follow the classical strategy of the analysis of finite element
  methods, which is detailed for a setting very close to ours in
  \cite{ehrlacher13} in various approximation proofs.

  
  {\it 1. Quasi-best approximation: } Following \cite[Lemma 7.3]{ehrlacher13},
  we can construct $T_R\bar{u}\in\Adm_0(R)$ such that, for any $\gamma > 0$ and
  for $R$ sufficiently large,
  \begin{eqnarray*}\label{proof-4-3-1}
    \|DT_R\bar{u}-D\bar{u}\|_{\ell^2_\gamma} 
    \leq C \|D\bar{u}\|_{\ell^2_\gamma(\Lambda\backslash B_{R/2})}
    \leq CR^{-d/2}
  \end{eqnarray*}
  where Theorem \ref{th:appl:regularity} is used for the last inequality.  We
  now fix some $r > 0$ such that $B_r(\bar{u}) \subset \Adm_{\frak{m}}$ for some
  $\frak{m} > 0$. Then, for $R$ sufficiently large, we have that
  $T_R \bar{u} \in B_{r/2}(\bar{u})$ and hence
  $B_{r/2}(T_R \bar{u}) \subset \Adm_{\frak{m}}$.
 
  Since $\E \in C^3(\Adm_0(R))$, $\delta\E$ and $\delta^2\E$ are Lipschitz
  continuous in $B_r(\bar{u}) \cap \Adm_0(R)$ with Lipschitz constants $L_1$ and
  $L_2$, that is,
  \begin{align}
    \label{proof-4-3-2}
    \|\delta\E(\bar{u})-\delta\E(T_R\bar{u})\| &\leq L_1\|D\bar{u}-DT_R(\bar{u})\|_{\ell^2_\gamma} \leq CR^{-d/2},
    \qquad \text{and} \\
    \label{proof-4-3-3}
    \|\delta^2\E(\bar{u})-\delta^2\E(T_R\bar{u})\| &\leq L_2\|D\bar{u}-DT_R(\bar{u})\|_{\ell^2_\gamma} \leq CR^{-d/2}.
  \end{align}

  {\it 2. Stability: } 
  Using \eqref{proof-2-6-2} and the facts that
  $v=0$ and $T_R\bar{u} = 0$ outside $B_R$,
  we have that there exists a constant $\gamma_{\rm s}$, such that
  \begin{equation}
    \label{eq:appl:consistency-error-hessian}
    \left| \big\<\big(\delta^2\E_R(T_R\bar{u})-\delta^2\E(T_R\bar{u})\big)v,v\big\> \right| \leq Ce^{-\gamma_{\rm s}\Rbuf} R^{d}\|Dv\|_{\ell^2_\gamma}^2.
  \end{equation}
  The proof of this identity is relatively straightforward but does require some
  details, which we present following the completion of the proof of the theorem.
Together with \eqref{eq:appl:strong-stab} and \eqref{proof-4-3-3} this leads to
\begin{align}
  \notag
  & \hspace{-0.5cm} \big\<\delta^2\E_R(T_R\bar{u})v,v\big\>  \\[1ex]
  \notag
  &= 
    \big\<\delta^2\E(\bar{u})v,v\big\> +
    \big\<\big(\delta^2\E(T_R\bar{u})-\delta^2\E(\bar{u})\big)v,v\big\> +
    \big\<\big(\delta^2\E_R(T_R\bar{u})-\delta^2\E(T_R\bar{u})\big)v,v\big\>  
  \\[1ex]
  &\geq 
    \big( \bar{c} - C( R^{-d/2} + 
    e^{-\gamma_{\rm s}\Rbuf}R^{d}) \big) \|Dv\|_{\ell^2_\gamma}^2
    \geq \frac{\bar{c}}{2}\|Dv\|_{\ell^2_\gamma}^2  
    \qquad\forall~v \in \Adm_0(R), 
    \label{proof-stability}
\end{align}
for sufficiently large $R$ and sufficiently large $c_{\rm buf}$.
  
{\it 3. Consistency: }
Similarly to \eqref{eq:appl:consistency-error-hessian}, 
we can derive that there exists a constant $\gamma_{\rm c}$ such that
\begin{equation}
  \label{eq:appl:consistency-error-forces}
  \left|\big\<\delta\E_R(T_R\bar{u})-\delta\E(T_R\bar{u}),v\big\>\right|
  \leq Ce^{-\gamma_{\rm c}\Rbuf} R^{d-1/2} \|Dv\|_{\ell^2_{\gamma}}.
\end{equation}
We also present the detailed proof of \eqref{eq:appl:consistency-error-forces} after the proof of this theorem.

In order to ensure that the truncation of the electronic structure (the
variational crime committed upon replacing $\E$ with $\E_R$) we must choose
$\Rbuf$ such that $e^{-\gamma_{\rm c} \Rbuf} R^{d-1/2} \leq C R^{-d/2}$, or
equivalently, $e^{-\gamma_{\rm c} \Rbuf} \leq C R^{-(3d+1)/2}$. On taking
logarithms, we observe that this is true provided that
$\Rbuf \geq c_{\rm buf} \log R$ for $c_{\rm buf}$ sufficiently large.

Next, employing \eqref{proof-4-3-1} we obtain that 
\begin{multline}\label{proof-consistency} 
\big\<\delta\E_R(T_R\bar{u}),v\big\>  =
\big\<\delta\E_R(T_R\bar{u})-\delta\E(T_R\bar{u}),v\big\> +
\big\<\delta\E(T_R\bar{u})-\delta\E(\bar{u}),v\big\>  \\[1ex]
\leq C\big(e^{-\gamma_{\rm c}\Rbuf} R^{d-1} + R^{-d/2}  \big)\|Dv\|_{\ell^2_\gamma}
\leq CR^{-d/2}\|Dv\|_{\ell^2_\gamma} \qquad\forall~v\in \Adm_0(R) \quad
\end{multline}
for sufficiently large $R$ and appropriate $c_{\rm buf}$.
 
{\it 4. Application of inverse function theorem: } With the stability
\eqref{proof-stability} and consistency \eqref{proof-consistency}, the inverse
function theorem \cite[Lemma B.1]{luskin13} implies the existence of $\bar{u}_R$
and the estimate \eqref{estimate-Du}.

{\it 5. Error in the energy: } For the estimate of energy difference functional,
we have from $\E\in C^2(\Adm_0)$ that
\begin{multline}\label{proof-4-3-5}
  \big| \E(\bar{u}_R) - \E(\bar{u}) \big| = \Big| \int_0^1 \big\< \delta\E\big(
  (1-s)
  \bar{u} + s \bar{u}_R\big), \bar{u}_R - \bar{u} \big\> \dd s \Big| \\
  = \Big| \int_0^1 \big\< \del\E\big( (1-s) \bar{u} + s \bar{u}_R\big) -
  \del\E(\bar{u}), \bar{u}_R - \bar{u} \big\> \dd s \Big| \leq C \| D\bar{u}_R -
  D \bar{u} \|_{\ell^2_\gamma}^2 \leq CR^{-d} .
\end{multline}
Using \eqref{proof-2-3-1} and $\bar{u}^R=0$ in $\L\backslash B_R$, there exists some constant $\gamma_{\rm e}$ such that
\begin{align}\label{proof-4-3-6}
\nonumber
& |\E(\bar{u}_R)-\E_R(\bar{u}_R)| \\[1ex]
\nonumber
&= \sum_{\ell\in \Lambda\cap B_{R+\Rbuf/2}} 
\left( E_{\ell}(x+\bar{u}^R)
- E^{\L_R}_{\ell}(x+\bar{u}^R)
+ E^{\L_R}_{\ell}(x) - E_{\ell}(x) \right) \\
\nonumber
& \qquad +\sum_{\ell\in \Lambda\backslash B_{R+\Rbuf/2}} 
\left( E_{\ell}(x+\bar{u}^R)-E_{\ell}(x) \right)
+\sum_{\ell\in \Lambda_R\backslash B_{R+\Rbuf/2}} 
\left( E^{\L_R}_{\ell}(x+\bar{u}^R)
-E^{\L_R}_{\ell}(x) \right) \\
\nonumber
&\leq C\left( \sum_{\ell\in \Lambda\cap B_{R+\Rbuf/2}} e^{-\gamma_{\rm e}(R+\Rbuf-|\ell|)}
+ \sum_{\ell\in \Lambda\backslash B_{R+\Rbuf/2}}
e^{-\gamma_{\rm e}(|\ell|-R)} \right) \\[1ex]
&\leq C R^{d-1}e^{-\gamma_{\rm e}\Rbuf/2}.
\end{align}
Using \eqref{proof-4-3-5}, \eqref{proof-4-3-6}, 
and possibly choosing a larger constant $c_{\rm buf}$,
we obtain \eqref{estimate-energy} for sufficiently large $R$.
\end{proof}

\begin{proof}[Proof of \eqref{eq:appl:consistency-error-hessian}]
  Let $u := T_R \bar{u}$ and
  $V_\ell^{\Omega} := V_\ell^{\Omega}(Du(\ell)) :=
  E_{\ell}^{\Omega}(x+u)-E_{\ell}^{\Omega}(x)$ and $V_\ell := V_\ell^\L$.
  Using $u =0$ in $\L \setminus B_R$ we have
	\begin{align*}
	& \big\<\big(\delta^2\E_R(u) - \delta^2\E(u)\big)v,v\big\>
	~=~ \sum_{\ell\in \Lambda\cap B_{R}} \big\< \big(\delta^2 V_{\ell}^{\Lambda_R} 
	- \delta^2 V_{\ell} \big) Dv(\ell), Dv(\ell) \big\>  \\
	& \qquad \qquad + \sum_{\ell\in \Lambda_R\backslash B_{R}} \big\< \big(\delta^2 V_{\ell}^{\Lambda_R} 
	- \delta^2 V_{\ell} \big) Dv(\ell), Dv(\ell) \big\> 
	- \sum_{\ell\in \Lambda\backslash 
		\Lambda_R} \big\<(\delta^2 V_{\ell})Dv(\ell), Dv(\ell)\big\> \\[1ex]
	&\qquad  =: {\rm T}_1 + {\rm T}_2 + {\rm T}_3.
	\end{align*}
	It is obvious from \eqref{proof-2-6-2} that
	\begin{align*}
          T_1 &\leq C\sum_{\ell\in \L\cap B_R}
                \sum_{\rho\in\L-\ell}\sum_{\sigma\in\L-\ell} 
                e^{-\gamma(R+\Rbuf-|\ell| + |\rho| + |\sigma|)} 
                |D_{\rho}v(\ell)| |D_{\sigma}v(\ell)| \\
              &\leq Ce^{-\gamma\Rbuf}\|Dv\|_{\ell^2_{\gamma/2}(\ell\in \L\cap B_R)}^2
	\end{align*}
	with $\gamma=\min\{\gamma_{\rm r},\kappa_0,\kappa_1,\kappa_2\}/4$.
	Using $v=0$ in $\L \setminus B_R$, we have
	from \eqref{proof-2-6-2} that
	\begin{align*}
	{\rm T}_2
	&\leq C \sum_{\ell\in \Lambda_R\backslash B_R}\sum_{\rho\in\Lambda-\ell\atop\ell+\rho\in B_R}
	\sum_{\sigma\in\Lambda-\ell\atop\ell+\sigma\in B_R}
	e^{-\gamma(R+\Rbuf-|\ell|+|\rho|+|\sigma|)} |D_{\rho}v(\ell)||D_{\sigma}v(\ell)|  \\
	&\leq  C e^{-\gamma(R+\Rbuf)} 
	\sum_{\ell\in \Lambda_R\backslash B_R}
	\left( \sum_{m\in \L \cap B_R}\sum_{n\in \L \cap B_R} e^{\gamma(2|\ell|-|\ell-m|-|\ell-n|)} \right)^{1/2}
	|Dv(\ell)|_{\gamma/2}^2 \\
	&\leq C e^{-\gamma(R+\Rbuf)}
	\left( \sum_{m\in \Lambda\cap B_R} 
	\sum_{n\in \L\cap B_R}  e^{\gamma(|m|+|n|)} \right)^{1/2} 
	\|Dv\|^2_{\ell^2_{\gamma/2}(\ell\in \Lambda_R\backslash B_R)} \\[1ex]
	&\leq C 
	e^{-\gamma\Rbuf} R^{d-1}
	\|Dv\|^2_{\ell^2_{\gamma/2}(\ell\in \Lambda_R\backslash B_R)},
	\end{align*}
	and from \eqref{site-locality-tdl} that 
	\begin{align*}
	{\rm T}_3
	&\leq C \sum_{\ell\in \Lambda\backslash \L_R}\sum_{\rho\in\Lambda-\ell\atop\ell+\rho\in B_R}
	\sum_{\sigma\in\Lambda-\ell\atop\ell+\sigma\in B_R} e^{-\eta_2(|\rho|+|\sigma|)} |D_{\rho}v(\ell)||D_{\sigma}v(\ell)|  \\
	&\leq  C \sum_{\ell\in \Lambda\backslash \L_R}
	\left( \sum_{m\in \Lambda\cap B_R} 
	\sum_{n\in \L\cap B_R}  e^{-\eta_2(|\ell-m|+|\ell-n|)} \right)^{1/2} |Dv(\ell)|_{\eta_2/2}^2 \\
	&\leq C
	\left( \sum_{m\in \Lambda\cap B_R} 
	\sum_{n\in \L\cap B_R} e^{-2\eta_2\Rbuf} \right)^{1/2} 
	\|Dv\|^2_{\ell^2_{\eta_2/2}(\Lambda\backslash\L_R)} \\[1ex]
	&\leq C 
	e^{-\eta_2\Rbuf} R^d
	\|Dv\|^2_{\ell^2_{\eta_2/2}(\Lambda\backslash\L_R)}.
	\end{align*}
	The estimates for $T_1$, $T_2$ and $T_3$ together yields
        \eqref{eq:appl:consistency-error-hessian} with
        $\gamma_{\rm s} = \min\{\gamma,\eta_2\}$. 
\end{proof}

\begin{proof}[Proof of \eqref{eq:appl:consistency-error-forces}]
  We continue to adopt the notation from the proof of
  \eqref{eq:appl:consistency-error-hessian}. 
  Using $u, v = 0$ in $\L \setminus B_R$ we have
  \begin{align*}
    \big\<\delta\E_R(u) - \delta\E(u),v\big\>
    &= \sum_{\ell\in \Lambda\cap B_{R}} \big\<\delta V_{\ell}^{\Lambda_R} 
      - \delta V_{\ell}, Dv(\ell)\big\>  
 + \sum_{\ell\in \Lambda_R\backslash B_{R}} \big\< \delta V^{\Lambda_R}_{\ell}
      - \delta V_{\ell}, Dv(\ell)\big\> \\
    & \qquad - \sum_{\ell\in \Lambda\backslash 
    	\Lambda_R} \big\<\delta V_{\ell}, Dv(\ell)\big\> \\
    &=: {\rm T}_1 + {\rm T}_2 + {\rm T}_3.
  \end{align*}
  From \eqref{proof-2-6-1} it follows that, for $\ell \in \L \cap B_R$,
  $|V_{\ell,\rho}^{\L_R} - V_{\ell,\rho}| \leq C e^{-\gamma(R+\Rbuf-|\ell|+|\rho|)}$ 
  with $\gamma=\min\{\gamma_{\rm r},\kappa_0,\kappa_1\}/2$,
  and hence,
  \begin{align*}
    {\rm T}_1 
    &\leq C \sum_{\ell \in \L_R \setminus B_R} \sum_{\rho \in \L-\ell}
      e^{-\gamma(R+\Rbuf-|\ell|+|\rho|)} |D_\rho v(\ell)| \\
    &\leq  C \sum_{\ell\in \Lambda\cap B_{R}} e^{-\gamma(R+\Rbuf-|\ell|)} 
      |Dv(\ell)|_{\gamma/2} \\
    &\leq C \bigg( \sum_{\L \cap B_R} 
       e^{-2\gamma(R+\Rbuf-|\ell|)} \bigg)^{1/2} \| Dv \|_{\ell^2_{\gamma/2}(\ell\in \Lambda\cap B_R)}  \\
    &\leq C e^{-\gamma \Rbuf} R^{(d-1)/2}  \| Dv \|_{\ell^2_{\gamma/2}(\Lambda\cap B_R)}.
  \end{align*}
  Similarly, using $v=0$ in $\L \setminus B_R$, we have from \eqref{proof-2-6-1}
  that
  \begin{align*}
  {\rm T}_2
  &\leq C \sum_{\ell\in \Lambda_R\backslash B_R}\sum_{\rho\in\Lambda-\ell\atop\ell+\rho\in B_R} e^{-\gamma(R+\Rbuf-|\ell|+|\rho|)} |D_{\rho}v(\ell)|  \\
  &\leq  C e^{-\frac{\gamma}{2}(R+\Rbuf)} 
  \left( \sum_{\ell\in \Lambda_R\backslash B_R} \sum_{k\in B_R} e^{\gamma(|\ell|-|\ell-k|)} \right)^{1/2}
  \left( \sum_{\ell\in \Lambda_R\backslash B_R} \sum_{\rho\in\Lambda-\ell\atop\ell+\rho\in B_R} e^{-\gamma|\rho|}|D_{\rho}v(\ell)|^2 \right)^{1/2} \\
  &\leq C e^{-\frac{\gamma}{2}(R+\Rbuf)}
  \left( \sum_{\ell\in \Lambda_R\backslash B_R} 
  \sum_{k\in B_R}  e^{\gamma|k|} \right)^{1/2} 
  \|Dv\|_{\ell^2_{\gamma/2}(\Lambda_R\backslash B_R)} \\[1ex]
  &\leq C 
  e^{-\frac{\gamma}{2}\Rbuf} R^{d-1} \big(\Rbuf\big)^{1/2}
  \|Dv\|_{\ell^2_{\gamma/2}(\Lambda_R\backslash B_R)}.
  \end{align*}
  Finally, from \eqref{site-locality-tdl} we obtain 
\begin{align*}
{\rm T}_3
&\leq C \sum_{\ell\in \Lambda\backslash \L_R}\sum_{\rho\in\Lambda-\ell\atop\ell+\rho\in B_R} e^{-\eta_1|\rho|} |D_{\rho}v(\ell)|  \\
&\leq  C 
\left( \sum_{\ell\in \Lambda\backslash\L_R} \sum_{k\in B_R} e^{-\eta_1 |\ell-k|} \right)^{1/2}
\left( \sum_{\ell\in \Lambda\backslash \L_R} \sum_{\rho\in\Lambda-\ell\atop\ell+\rho\in B_R} e^{-\eta_1|\rho|}|D_{\rho}v(\ell)|^2 \right)^{1/2} \\
&\leq C
\left( \sum_{k\in B_R} \sum_{\ell\in \Lambda\backslash\L_R} e^{-\eta_1(|\ell|-R)} \right)^{1/2} 
\|Dv\|_{\ell^2_{\eta_1/2}(\Lambda\backslash\L_R)} \\[1ex]
&\leq C 
e^{-\frac{\eta_1}{2}\Rbuf} R^{d-1/2}
\|Dv\|_{\ell^2_{\eta_1/2}(\Lambda\backslash\L_R)}.
\end{align*}
The estimates for $T_1$, $T_2$ and $T_3$ together yields
\eqref{eq:appl:consistency-error-forces} with $\gamma_{\rm c} =
\min\{\gamma,\eta_1\}/2$.
\end{proof}

\begin{remark}
  \label{rem:buffer radius}
  The choice of buffer width $\Rbuf$ is the most interesting aspect of
  Theorem~\ref{th:appl:min-approx}. As expected from the exponential
  localisation results, we obtain that $\Rbuf$ should be proportional to
  $\log(R)$. The fact that the constant of proportionality is important makes an
  implementation difficult. At least according to our proof, if we were to
  choose $\Rbuf = c \log(R)$ with a small $c$, then we would obtain a reduced
  convergence rate. 
  
  Our numerical results in \S~\ref{sec:convergence rate} show no such
  dependence, which may indicate that our proof is in fact sub-optimal, however
  it is equally possible that this effect can only be observed for much larger
  system sizes than we are able to simulate.
  %
\end{remark}

\section{Numerical results}
\label{sec:appl:numerics}
%
%
%
We present numerical experiments to illustrate the results of the paper: (1) the
locality of the site energies and (2) the convergence of the truncation scheme
described in \S~\ref{sec:appl:truncation}. Given the present paper is primarily
concerned with the analytical foundations, we will show only a limited set of
results, employing a highly simplified toy model. We will present more
comprehensive numerical results in the companion papers \cite{chenpreE,
  chenpreF}. All numerical experiments were carried out in {\sc Julia}
\cite{Julia}.

\subsection{Toy Model}
\label{sec:toyu_model}
The Hamiltonian matrix is given by
\begin{align*}
  H_{\ell m}(y) = h(|y_\ell-y_m|) \quad \text{where} \quad
  h(r) &= \big( e^{-2\alpha(r-r_0)} - 2 e^{-\alpha(r-r_0)} \big) f_{\rm cut}(r), \\
  \text{and} \quad f_{\rm cut}(r) &= 
    \begin{cases}
      \big(1+e^{1/(r-r_{\rm cut})}\big)^{-1} &, r < r_{\rm cut}, \\
      0, & r \geq r_{\rm cut},
    \end{cases}
\end{align*}
with model parameters $\alpha = 2.0, r_0 = 1.0, r_{\rm cut} = 2.8$. The pair
potential term is set to be zero. 

Numerical tests suggest that a triangular lattice $A_{\rm tri} \Z^d$ with
\begin{displaymath}
  A_{\rm tri} = s \left(\begin{matrix}
      1 & 1/2 \\ 0 & \sqrt{3}/2
    \end{matrix} \right)
\end{displaymath}
and $s$ a scaling factor (close to $1.0$) is a stable equilibrium in the sense
of \S~\ref{sec:tb model for point defects}.

The 2D setting and the single orbital per site make this a convenient setting
for preliminary numerical tests.

\subsection{Locality of the site energy}
\label{sec:locality}
We construct two test configurations: (1) We ``carve'' a finite lattice domain
$\L_R = B_R \cap A_{\rm tri} \Z^d$ from the triangular lattice, and perturb each
position $y \in \L_R$ by a vector with entries equidistributed in $[0, 0.1]$ to
obtain $y$.  (2) We obtain a second test configuration by removing some random
lattice sites from $\L_R$ (vacancies) and perturb the remaining positions as in
(1) to obtain $y$. We then compute the first and second site energy derivatives
$E_{0, m}(y)$ and $E_{0, mn}(y)$ and plot them against, respectively, $r_{0m}$
and $r_{0m}+r_{0n}$.

In the test shown in Figure \ref{fig:locality}, we chose $s = 1.0$ and
$R = 10.0$, and the sites removed in (2) are
$A_{\rm tri} (1,0), A_{\rm tri} (0, -3), A_{\rm tri} (-2,2)$ and
$A_{\rm tri} (2,5)$. We clearly observe the predicted exponential decay. 

\begin{figure}
  \centering
  \includegraphics[width=13cm]{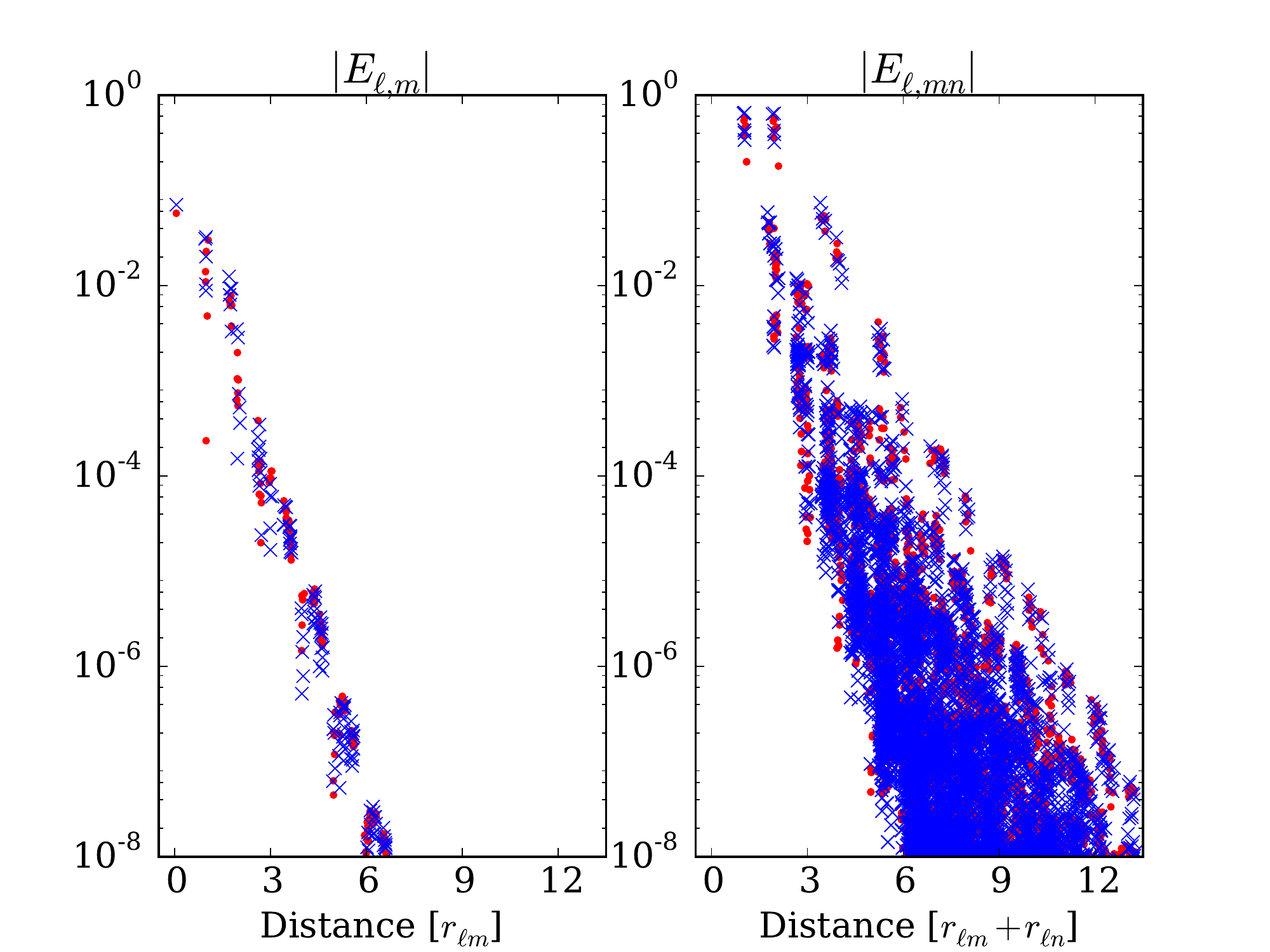}
  \caption{Locality of the site energy for the tight-binding toy model described
    in \S~\ref{sec:toyu_model}. Red dots denote configuration (1), while blue
    crosses denote configuration (2).}
  \label{fig:locality}
\end{figure}

\subsection{Convergence rate}
\label{sec:convergence rate}
In our second numerical experiment we confirm the prediction of Theorem
\ref{th:appl:min-approx}. We adopt again the model from
\S~\ref{sec:toyu_model}. As reference configuration we choose a di-vacancy
configuration,
\begin{displaymath}
  \L = A_{\rm tri} \Z^2 \setminus \big\{ (0,0), (1,0)\big\}.
\end{displaymath}
Then, for increasing radii $R$ with associated buffer radii $\Rbuf$,
\begin{equation}
  \label{eq:R_Rbuf}
  \begin{array}{r|c|lllll}
    & R & 3 & 4 & 6 & 8 & 11 \\
    \hline
    \text{Set 1}& \Rbuf & 2.1 & 2.4 & 2.8 & 3.0 & 3.4 \\
    \text{Set 2} & \Rbuf & 1.0 & 1.7 & 1.7 & 2.0 & 2.0 \\
    \text{Set 3} & \Rbuf & 1.0 & 1.0 & 1.7 & 1.7 & 2.0  
  \end{array}
\end{equation}
we solve the problem \eqref{eq:appl:min-approx}. In {\it Set 1} we have chosen
$\Rbuf = 1 + \log(R)$, while in {\it Sets 2, 3} we have chosen smaller buffer
radii to investigate the effect of these choices on the error in the numerical
solution.

The computed solutions $\bar{u}_R$ are compared against a high-accuracy solution
with $R = 20, \Rbuf = 11$, which yields the convergence graphs displayed in
Figure~\ref{fig:errors}, fully confirming the analytical prediction. To measure
errors, instead of $\|D\cdot\|_{\ell^2_\gamma}$, we employ the equivalent norm
$\| D^{\rm nn} \cdot \|_{\ell^2}$ with
$D^{\rm nn} u = ( D_{\rho} u)_{\rho \in \pm A_{\rm tri} e_i, i = 1, 2}$.

We do not observe a pronounced buffer size effect. This could have a number of
reasons, such as the fact that we are not far enough in the asymptotic regime or
simply that the model we are employing is ``too local'' to observe this. We will
present more extensive numerical results with a wider variety of TB models in
\cite{chenpreE, chenpreF}.

\begin{figure}
  \centering
  \includegraphics[width=13cm]{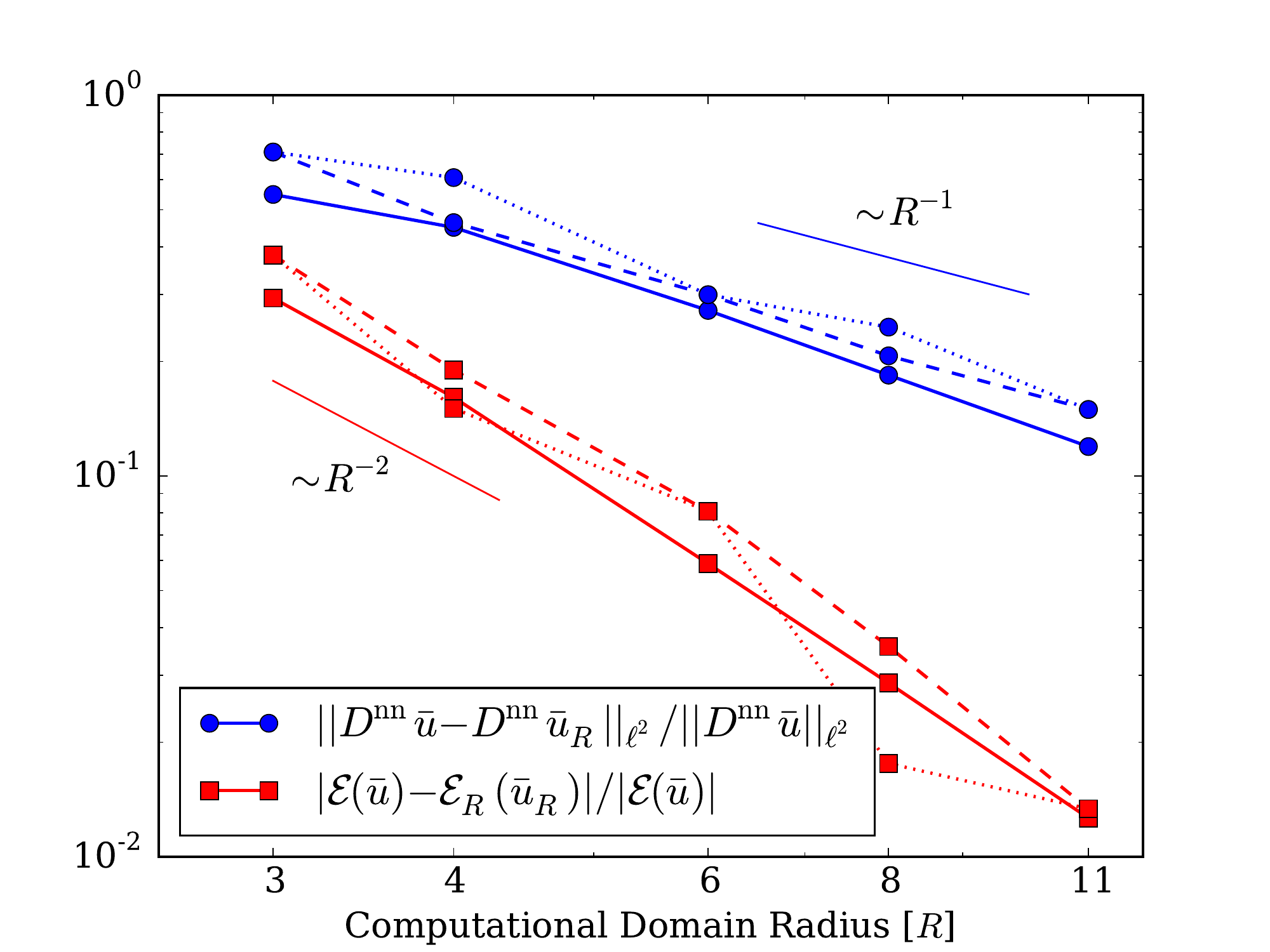}
  \caption{Convergence of \eqref{eq:appl:min-approx} with increasing domain size
    $R$, as described in \S~\ref{sec:convergence rate}. Set 1: full lines; Set
    2: dashed lines; Set 3: dotted lines.}
  \label{fig:errors}
\end{figure}

\section{Concluding remarks} \label{sec-conclusions}
\setcounter{equation}{0}
The main purpose of this paper was to set the scene for a rigorous numerical
analysis approach to QM/MM coupling. We have achieved this by developing a new
class of locality results for the tight binding model. Precisely, we have shown
that the total band energy can be decomposed into contributions from individual
sites in a meaningful, i.e. local, way.

This strong locality result is the basis for extending the theory of crystalline
defects of \cite{ehrlacher13}, which we have hinted at in \S~\ref{sec:tb model
  for point defects}, and carried out in detail in \cite{chenpre_vardef}. In
further forthcoming articles \cite{chenpreE,chenpreF} we are employing it to
develop new QM/MM coupling schemes for crystalline defects as well as their
rigorous analysis.

A key question that remains to be investigated in the future, is whether our
locality results extend to more accurate electronic structure models such as
Kohn--Sham density functional theory. Understanding this extension is critical to
take the theory we are developing in the present article and in
\cite{chenpreE,chenpreF} towards materials science applications. However, there
are many technical issues arising from the nonlinearity, the continuous nature,
and in particular the long-range Coulomb interaction.

\appendix
\renewcommand\thesection{\appendixname~\Alph{section}}

\section{Multiple orbitals per atom} 
\label{sec-appendix-multiorbital}
\renewcommand{\theequation}{A.\arabic{equation}}
\renewcommand{\thetheorem}{A.\arabic{theorem}}
\renewcommand{\thelemma}{A.\arabic{lemma}}
\renewcommand{\theproposition}{A.\arabic{proposition}}
\renewcommand{\thealgorithm}{A.\arabic{algorithm}}
\renewcommand{\theremark}{A.\arabic{remark}}
\setcounter{equation}{0}

We have assumed in \S~\ref{sec-tb-finite-hamiltonian}, and throughout this
paper, that there is only one atomic orbital for each atom ($n_{\Xi}$=1). In
this case, the symmetry assumption {\bf H.sym (i)} is natural.  However, in
practical calculations, there are multiple atomic-like orbitals
$\phi_{\ell\alpha}$ associated with each atomic site $\ell$.  Here, $\alpha$
denotes both the orbital and angular quantum numbers of the atomic state. Many
TB models employ one $s$-orbital $|s\>$, three $p$-orbitals
$\{ |x\>,~ |y\>,~ |z\> \}$ and five $d$-orbitals
$\{ |xy\>,~ |yz\>,~ |zx\>,~ |x^2-y^2\>,~ |3z^2-r^2\> \}$ per atom
\cite{finnis03,slater54}.  Except for the $s$-orbital, all other orbitals have a
spacial orientation, which means that the Hamiltonian matrix elements are {\em
  not} invariant under rotation/reflection. Therefore {\bf H.sym~(i)} is invalid
and must be reformulated.

We assume in the following that $n_{\Xi}>1$ and
$\{\phi_{\ell\alpha}\}_{1\leq\alpha\leq n_{\Xi}}$ is the set of atomic-like
orbitals for the site $\ell$.  The Hamiltonian can be expressed by
\eqref{tb-H-elements-abstract},
\begin{eqnarray}\label{tb-H-elements-appendix}
\Big(\mathcal{H}(y)\Big)_{\ell k}^{\alpha\beta} = 
\int_{\mathbb{R}^d} \phi_{\ell\alpha}({\bf r}-y(\ell))\widehat{\mathcal{H}}(y)
\phi_{k\beta}({\bf r}-y(k)) \dd {\bf r}.
\end{eqnarray}
Applying an isometry $g$ to $y$, we obtain
\begin{eqnarray}\label{tb-H-elements-isometry}
\Big(\mathcal{H}\big(g(y)\big)\Big)_{\ell k}^{\alpha\beta} = 
\int_{\mathbb{R}^d} \phi_{\ell\alpha}\big({\bf r}-g(y(\ell))\big)
\widehat{\mathcal{H}}\big(g(y)\big) \phi_{k\beta}\big({\bf r}-g(y(k))\big) 
\dd{\bf r}.
\end{eqnarray}
For simplicity of notation, we define
\begin{eqnarray*}
\psi_{\ell\alpha}({\bf r}) = \phi_{\ell\alpha}\Big( g^{-1}({\bf r})-y(\ell) \Big)
\quad{\rm and}\quad
\varphi_{\ell\alpha}({\bf r}) = \phi_{\ell\alpha}\Big( {\bf r}-g\big(y(\ell)\big) \Big).
\end{eqnarray*}
We assume that the two sets of atomic orbitals
$\{\psi_{\ell\alpha}\}_{1\leq\alpha\leq n_{\Xi}}$ and
$\{\varphi_{\ell\alpha}\}_{1\leq\alpha\leq n_{\Xi}}$ span the same subspace.
This is true for almost all tight-binding models since 
the set of the atomic orbitals always include
all three $p-$orbitals (if the $p-$orbital is involved) and 
all five $d-$orbitlas (if the $d-$orbital is involved).
%
Then, there exists an orthogonal matrix $Q^{\ell}\in\R^{n_{\Xi}\times n_{\Xi}}$
such that
$\varphi_{\ell\alpha} = \sum_{1\leq\beta\leq
  n_{\Xi}}Q^{\ell}_{\alpha\beta}\psi_{\ell\beta}$.
We have from \eqref{tb-H-elements-isometry} that
\begin{eqnarray}\label{ham-local-tran} \nonumber
\Big(\mathcal{H}\big(g(y)\big)\Big)_{\ell k}^{\alpha\beta} &=& 
\int_{\mathbb{R}^d} \varphi_{\ell\alpha}({\bf r})
\widehat{\mathcal{H}}\big(g(y)\big) \varphi_{k\beta}({\bf r}) \dd{\bf r} \\[1ex]
&=& \sum_{1\leq\alpha',\beta'\leq n_{\Xi}}
Q^{\ell}_{\alpha\alpha'} Q^{k}_{\beta\beta'}
\int_{\mathbb{R}^d} \psi_{\ell\alpha'}({\bf r})
\widehat{\mathcal{H}}\big(g(y)\big) \psi_{k\beta'}({\bf r}) \dd{\bf r}.
\end{eqnarray}
Since $g$ is an isometry, it is natural to assume that
\begin{eqnarray}\label{ham-iso-assump}
\Big(\mathcal{H}(y)\Big)_{\ell k}^{\alpha\beta} =
\int_{\mathbb{R}^d} \psi_{\ell\alpha}({\bf r})
\widehat{\mathcal{H}}\big(g(y)\big) \psi_{k\beta}({\bf r}) \dd{\bf r}.
\end{eqnarray}

Let
$\Big(\mathcal{H}(y)\Big)_{\ell k} := \left[\Big(\mathcal{H}(y)\Big)_{\ell
    k}^{\alpha\beta}\right]_{1\leq\alpha,\beta\leq n_{\Xi}} \in\R^{n_{\Xi}\times
  n_{\Xi}}$
denote the local Hamiltonian, then we have from \eqref{ham-local-tran} and
\eqref{ham-iso-assump} that
\begin{eqnarray}\label{ham-local-relate}
\Big(\mathcal{H}\big(g(y)\big)\Big)_{\ell k} =
Q^{\ell} \cdot \Big(\mathcal{H}(y)\Big)_{\ell k} 
\cdot \left(Q^{k}\right)^{\rm T},
\end{eqnarray}
which yields
\begin{eqnarray}\label{ham-relate}
\mathcal{H}\big(g(y)\big) = Q \cdot
\mathcal{H}(y) \cdot Q^{\rm T}
\quad{\rm with}\quad Q={\rm diag}\left\{Q^1,\cdots,Q^N\right\}.
\end{eqnarray}
Note that $Q^{\ell}$ are orthogonal matrices, hence $Q$ is orthogonal as well.
Therefore, the spectrum of $\mathcal{H}(y)$ and $\mathcal{H}\big(g(y)\big)$ are
equivalent:
\begin{eqnarray}\label{appendix-proof-1-1}
\epsilon_s = \epsilon_s^g \qquad{\rm for}~~ 1\leq s\leq N\cdot n_{\Xi}.
\end{eqnarray}
Let 
\begin{eqnarray*}
\Psi_s = \left( \begin{array}{c}
\Psi_s^1 \\ \vdots \\ \Psi_s^N
\end{array} \right)
\quad{\rm with}\quad
\Psi_s^{\ell} = \left( \begin{array}{c}
\Psi_s^{\ell 1} \\ \vdots \\ \Psi_s^{\ell n_{\Xi}}
\end{array} \right)
\qquad{\rm for}~~ 1\leq s\leq N\cdot n_{\Xi}
\end{eqnarray*}
be the eigenfunction of $\mathcal{H}(y)$ corresponding to the eigenvalue $\epsilon_s$. 
Then the corresponding eigenfunction of $\mathcal{H}\big(g(y)\big)$ is
\begin{eqnarray}\label{appendix-proof-1-2}
\Psi_s^g = Q\Psi_s = \left( \begin{array}{c}
Q^1\Psi_s^1 \\ \vdots \\ Q^N\Psi_s^N
\end{array} \right) .
\end{eqnarray}

Next we note that, with multiple orbitals per atom, the expression
\eqref{site-energy} should be rewritten as
\begin{eqnarray}\label{site-energy-multi}
E_\ell(y) = \sum_{s} f(\varepsilon_s) \varepsilon_s \sum_{\alpha} \left(\Psi_s^{\ell\alpha}\right)^2 
= \sum_{s}\mathfrak{f}(\varepsilon_s) \sum_{\alpha} \left(\Psi_s^{\ell\alpha}\right)^2 .
\end{eqnarray}
Taking into account \eqref{appendix-proof-1-1}, \eqref{appendix-proof-1-2} and
\eqref{site-energy-multi} we obtain invariance of the site energy under
isometries,
\begin{eqnarray}
  E_{\ell}(y) = E_{\ell}(g(y)),
\end{eqnarray}

To summarize, in the case of multiple orbitals, the assumption {\bf H.sym (i)}
should become
\begin{flushleft} {\bf H.sym' (i)}.  If $y\in\mathcal{V}_{\mathfrak{m}}^N$ and
  $g:\R^d\rightarrow\R^d$ is an isometry on $\R^d$, then there exist orthogonal
  matrices $Q^{\ell}\in\R^{n_{\Xi}\times n_{\Xi}}$ for $\ell=1,\cdots,N$ such
  that \eqref{ham-relate} is satisfied.
\end{flushleft}
(This is equivalent to {\bf H.sym (i)} when $n_{\Xi}=1$.)

\begin{remark}
  Slater and Koster worked out expressions such as \eqref{ham-local-relate} and
  \eqref{ham-relate} for all integrals between $s$, $p$ and $d$ orbitals and
  presented them in Table 1 of their paper \cite{slater54}. This has been
  invaluable for practical calculations, see,
  e.g. \cite{finnis03,Papaconstantopoulos15}.
\end{remark}

\begin{remark}
  We stress again that all our assumptions and analysis in the present paper can
  be extended to the multi-orbital case without any difficulty, by taking the
  Hamiltonian as a block matrix with
  \begin{eqnarray}
    h_{\ell k}(y) = \Big(\mathcal{H}(y)\Big)_{\ell k} \in \R^{n_{\Xi}\times n_{\Xi}}
  \end{eqnarray}
  and $|h_{\ell k}(y)|$ the Frobenius norm of the submatrix.
\end{remark}

\section{Site energy with non-orthogonal orbitals} 
\label{sec-appendix-siteoverlap}
\renewcommand{\theequation}{B.\arabic{equation}}
\renewcommand{\thetheorem}{B.\arabic{theorem}}
\renewcommand{\thelemma}{B.\arabic{lemma}}
\renewcommand{\theproposition}{B.\arabic{proposition}}
\renewcommand{\thealgorithm}{B.\arabic{algorithm}}
\renewcommand{\theremark}{B.\arabic{remark}}
\setcounter{equation}{0}

We consider the tight binding model with non-orthogonal 
atomic orbitals in this appendix.
It has been shown in Remark \ref{rem:discussion_of_HX} (iv) that the transformed Hamiltonian is
\begin{eqnarray*}
	\widetilde{\Ham}=\mathcal{M}^{-1/2}\Ham\mathcal{M}^{-1/2}
\end{eqnarray*}
when the overlap matrix is not an identity matrix.
Then the transformed eigenvectors of $\Ham\psi_s=\varepsilon_s\psi_s$ become
$\tilde{\psi}_s = \mathcal{M}^{1/2}\psi_s$,
and following \eqref{E-El}, the site energy is given by
\begin{eqnarray}\label{El-B-1}
E_\ell(y) = \sum_{s} f(\varepsilon_s) \varepsilon_s [\tilde{\psi}_s]_{\ell}^2 
= \sum_{s}\mathfrak{f}(\varepsilon_s) [\mathcal{M}^{1/2}\psi_s]_{\ell}^2.
\end{eqnarray}
Since the square root of a matrix in \eqref{El-B-1} introduces additional
computational cost for the site energy computations, we modify the definition of
site energy in practice by
\begin{eqnarray}\label{El-B-2}
\widetilde{E}_\ell(y) = \sum_{s}\mathfrak{f}(\varepsilon_s)
 [\mathcal{M}\psi_s]_{\ell}[\psi_s]_{\ell}.
\end{eqnarray}
The following result states that the modified site energy
 \eqref{El-B-2} preserves the locality property.

\begin{lemma}\label{lemma-locality-site-B}
  Assume that \asL, \asHtb, \asHloc, {\bf F} are satisfied, and moreover, the
  overlap matrix $\mathcal{M}$ satisfy the same conditions as those in \asHtb,
  \asHloc.

	Then, for
	$1\leq j\leq\mathfrak{n}$, there exist positive constants $\tilde{C}_j$ and $\tilde{\eta}_j$
	such that for any $\ell\in\Lambda_N$,
	\begin{eqnarray}\label{site-j-decay-finite-B}
	\left|\frac{\partial^j \widetilde{E}_{\ell}(y)}{\partial [y(m_1)]_{i_1}\cdots\partial [y(m_j)]_{i_j}}\right|
	\leq \tilde{C}_j e^{-\tilde{\eta}_j\sum_{l=1}^j|y(\ell)-y(m_l)|} \qquad 1\leq i_1,\cdots,i_j\leq d.
	\end{eqnarray}
\end{lemma}

\begin{proof}
	Let $\varXi = \mathfrak{f}(\widetilde{\mathcal{H}})$
	 (with $\varXi_{jk} = \sum_s\mathfrak{f}(\varepsilon_s)[\tilde{\psi}_s]_j[\tilde{\psi}_s]_k$).
	The assumptions on $\Ham$ and $\mathcal{M}$ imply that the transformed Hamiltonian $\widetilde{\Ham}$ also satisfies the conditions in \asHtb and \asHloc.	
	Using Lemma \ref{lemma-resolvant-decay} and 
	similar arguments as those in the proof of Lemma \ref{lemma-locality-site},
	we have
	\begin{eqnarray}\label{proof-B-1}
          |\varXi_{jk}| \leq Ce^{-\gamma|y(j)-y(k)|}
	\quad{\rm and}\quad
          \bigg|\frac{\partial\varXi_{jk}}{\partial [y(n)]_i} \bigg| \leq Ce^{-\gamma(|y(n)-y(j)|+|y(n)-y(k)|)}
	\end{eqnarray}
	with some constants $C$ and $\gamma$.
	Similarly, the assumptions on $\mathcal{M}$ also imply
	\begin{eqnarray}\label{proof-B-2}
	\big|\mathcal{M}^{\pm 1/2}_{jk}\big| \leq Ce^{-\gamma|y(j)-y(k)|}
	\quad{\rm and}\quad
          \Bigg|\frac{\partial\mathcal{M}^{\pm 1/2}_{jk}}{\partial [y(n)]_i}\Bigg| \leq Ce^{-\gamma(|y(n)-y(j)|+|y(n)-y(k)|)}.
	\end{eqnarray}
	
	We have from \eqref{El-B-1} and \eqref{El-B-2} that
	\begin{align*}
	E_{\ell} &=  \sum_s\mathfrak{f}(\varepsilon_s)\sum_{jk}\mathcal{M}^{1/2}_{\ell j}\mathcal{M}^{1/2}_{\ell k}[\psi_s]_j[\psi_s]_k 
	= \varXi_{\ell\ell}, \\
	\widetilde{E}_{\ell} &=  \sum_s\mathfrak{f}(\varepsilon_s)\sum_{j}\mathcal{M}_{\ell j}[\psi_s]_j[\psi_s]_{\ell}
	= [\mathcal{M}^{1/2}\varXi\mathcal{M}^{-1/2}]_{\ell\ell}.
	\end{align*}
	Therefore,
	\begin{eqnarray*}
	\frac{\partial \widetilde{E}_{\ell}}{\partial [y(n)]_i}
	= \sum_{jk}\left(
	\frac{\partial\mathcal{M}^{1/2}_{\ell j}}{\partial [y(n)]_i}\varXi_{jk}\mathcal{M}^{-1/2}_{k\ell}
	+ \mathcal{M}^{1/2}_{\ell j}\frac{\partial\varXi_{jk}}{\partial [y(n)]_i}\mathcal{M}^{-1/2}_{k\ell}
	+ \mathcal{M}^{1/2}_{\ell j}\varXi_{jk}\frac{\partial\mathcal{M}^{-1/2}_{k\ell}}{\partial [y(n)]_i} \right),
	\end{eqnarray*}
	which together with \eqref{proof-B-1}, \eqref{proof-B-2},
	and a similar argument as that in \eqref{eq:site-force-estimate-proof} 
	completes the proof of \eqref{site-j-decay-finite-B} for $j=1$.
		
	The proofs for $2\leq j\leq \mathfrak{n}$ are similar.
\end{proof}

\bibliographystyle{siam}
\bibliography{qmmmtb1bib}

\end{document}